\documentclass[a4paper]{amsart}
\usepackage{amssymb}
\usepackage[color,all]{xy}
\usepackage{wrapfig}
\usepackage{secdot}
\usepackage{graphicx}
\usepackage{color}
\usepackage[all]{xy}
\usepackage{amscd}
\usepackage[top=35truemm,bottom=30truemm,right=30truemm,left=30truemm]{geometry}

\numberwithin{equation}{section}
\newtheorem{theorem}{Theorem}[section]
\newtheorem{proposition}[theorem]{Proposition}
\newtheorem{def-prop}[theorem]{Definition and Proposition}
\newtheorem{definition}[theorem]{Definition}
\newtheorem{lemma}[theorem]{Lemma}
\newtheorem{corollary}[theorem]{Corollary}
\newtheorem{example}[theorem]{Example}
\newtheorem{remark}[theorem]{Remark}

\newcommand{\stau}{\mathop{\mathrm{s\tau}\text{-}\mathsf{tilt}}\nolimits}

\newcommand{\Image}{\mathop{\mathrm{Im}}\nolimits}
\newcommand{\Ker}{\mathop{\mathrm{Ker}}\nolimits}
\newcommand{\tops}{\mathop{\mathrm{top}}\nolimits}
\newcommand{\rad}{\mathop{\mathrm{rad}}\nolimits}
\newcommand{\soc}{\mathop{\mathrm{soc}}\nolimits}
\newcommand{\End}{\mathop{\mathrm{End}}\nolimits}
\newcommand{\Ext}{\mathop{\mathrm{Ext}}\nolimits}

\newcommand{\Hom}{\mathop{\mathrm{Hom}}\nolimits}
\newcommand{\Kb}{\mathop{\mathsf{K^{\rm b}}}\nolimits}
\newcommand{\Ktwo}{\mathop{\mathsf{K^{\rm 2}_{\epsilon}}}\nolimits}
\newcommand{\twosilt}{\mathop{\mathrm{2}\text{-}\mathsf{silt}}\nolimits}
\newcommand{\twotilt}{\mathop{\mathrm{2}\text{-}\mathsf{tilt}}\nolimits}
\newcommand{\thick}{\mathop{\mathsf{thick}}\nolimits}

\newcommand{\proj}{\mathop{\mathsf{proj}}\nolimits}
\newcommand{\Filt}{\mathop{\mathsf{Filt}}\nolimits}
\newcommand{\Sub}{\mathop{\mathsf{Sub}}\nolimits}
\newcommand{\Fac}{\mathop{\mathsf{Fac}}\nolimits}
\newcommand{\module}{\mathop{\mathsf{mod}}\nolimits}
\newcommand{\tilt}{\mathop{\mathsf{tilt}}\nolimits}
\newcommand{\ftors}{\mathop{\mathsf{f}\text{-}\mathsf{tors}}\nolimits}
\newcommand{\Hasse}{\mathop{\mathsf{Hasse}}\nolimits}
\newcommand{\fators}{\mathop{\mathsf{fa}\text{-}\mathsf{tors}}\nolimits}
\newcommand{\ffators}{\mathop{\mathsf{ffa}\text{-}\mathsf{tors}}\nolimits}
\newcommand{\add}{\mathop{\mathsf{add}}\nolimits}
\newcommand{\tors}{\mathop{\mathsf{tors}}\nolimits}
\newcommand{\coker}{\mathop{\mathrm{coker}}\nolimits}

\title[Torsion classes for algebras with radical square zero]{Classifying torsion classes for algebras with radical square zero via sign decomposition}
\author{TOSHITAKA AOKI} 

\address{
Graduate School of Mathematics, Nagoya University, Frocho, Chikusaku, Nagoya 464-8602, Japan
}
\email{m15001d@math.nagoya-u.ac.jp} 
\begin{document}

\begin{abstract}
To study the set of torsion classes of a finite dimensional basic algebra, we use a decomposition, called sign-decomposition, parametrized by elements of $\{\pm1\}^n$ where $n$ is the number of simple modules. 
If $A$ is an algebra with radical square zero, then for each $\epsilon \in \{\pm1\}^n$ 
there is a hereditary algebra $A_{\epsilon}^!$ with radical square zero and a bijection between the set of torsion classes of $A$ associated to $\epsilon$ and the set of faithful torsion classes of $A_{\epsilon}^!$. 
Furthermore, this bijection preserves the property of being functorially finite. 
As an application in $\tau$-tilting theory, we prove that the number of support $\tau$-tilting modules over Brauer line algebras (resp. Brauer odd-cycle algebras) having $n$ edges is $\binom{2n}{n}$ (resp. $2^{2n-1}$). 
\end{abstract}

\maketitle

\section{Introduction}
The importance of torsion classes in representation theory of 
finite dimensional algebras is increasing due to the recent development of $\tau$-tilting theory \cite{AIR, DIJ} and the study of Bridgeland's stability spaces on derived categories \cite{Bri}.

A tilting module plays the role of a generator (Ext-projective module) 
of the corresponding torsion class. 
This viewpoint is adopted by Adachi-Iyama-Reiten for $\tau$-tilting theory, 
which results a bijection between support $\tau$-tilting modules (or $\tau$-tilting pairs) 
and functorially finite torsion classes. 
This bijection gives a partial order on the set of functorially finite torsion classes, making them easier to study. 
Support $\tau$-tilting modules, and thereby functorially finite torsion classes, are classified for several classes of algebras \cite{Ada2, AM, EJR}. 
In contrast, it is very hard to classify non-functorially finite torsion classes in general.

This paper has two principal aims. The first is to introduce a new approach, which we call sign-decomposition,  
to the classification problem of all torsion classes.
Let $A$ be a finite dimensional basic algebra over a field $k$ having $n$ simple modules $S(1), \ldots, S(n)$ which are mutually non-isomorphic. 
Clearly, the set $\tors A$ of torsion classes of $A$ is decomposed into subsets 
$\tors_{\epsilon} A:=\{\mathcal{T} \in \tors A \mid S(i) \in \mathcal{T} \Leftrightarrow \epsilon(i)=1\}$ for all $\epsilon \in \{\pm1\}^n$. 
In Proposition \ref{restriction-epsilon}, we show 
that there is a naturally defined bijection between 
\begin{itemize}
\item $\ftors_{\epsilon} A$ the set of functorially finite torsion classes in $\tors_{\epsilon} A$.  
\item $\stau_{\epsilon} A$ the set of isomorphism classes of $\tau$-tilting pairs for $A$ whose $g$-vectors lie in 
$\{ x \in \mathbb{Z}^n \mid x_i \in \epsilon(i) \cdot \mathbb{Z}_{>0}, i=1,\ldots,n \}$.
\end{itemize}

The second aim of this paper is to classify all torsion classes for algebras $A$ with radical square zero, as an application of sign-decomposition. 
In representation theory of algebras with radical square zero, a crucial role is played by the stable equivalence between $A$ and the hereditary algebra
$
\begin{pmatrix} 
A/J_A & J_A \\
0  & A/J_A
\end{pmatrix}$
with radical square zero, where $J_A$ is the Jacobson radical of $A$. 
Using this equivalence, one can characterize when $A$ is representation finite \cite{Ga, DR}. 
In this paper, we provide a new link between algebras with radical square zero and hereditary algebras. 
To each $\epsilon \in \{\pm1\}^n$, we associate the hereditary algebra $A_{\epsilon}^!$ 
with radical square zero and show the following bijection. 
Now, we denote by $\fators A$ the set of faithful torsion classes of $A$, 
by $\tilt A$ the set of isomorphism classes of basic tilting $A$-modules.

\begin{theorem}[Theorem \ref{tors-fators} and Theorem \ref{stau-tilt-tilt}] \label{Theorem A}
Let $A$ be a finite dimensional basic algebra with radical square zero having $n$ simple modules $S(1),\ldots, S(n)$. 
For each $\epsilon \in \{\pm1\}^n$, 
there is a hereditary algebra $A_{\epsilon}^!$ with radical square zero 
and a bijection between
\begin{enumerate}
\item $\tors_{\epsilon} A$ and 
\item $\fators A_{\epsilon}^!$.
\end{enumerate} 
Moreover, it induces a bijection between
\begin{enumerate}
\item[(1)'] $\stau_{\epsilon} A$ and 
\item[(2)'] $\tilt A_{\epsilon}^!$.
\end{enumerate}
\end{theorem}

In $\tau$-tilting theory, the latter bijection is useful for analyzing the structure of support $\tau$-tilting modules. 
For example, the $g$-vector of a support $\tau$-tilting $A$-module can be computed from the dimension vector of the corresponding tilting $A_{\epsilon}^!$-module 
(Theorem \ref{transformation}). The Hasse quiver of the set of support $\tau$-tilting modules is completely described by using the Hasse quiver of $\tilt A_{\epsilon}^!$ for all $\epsilon \in \{\pm1\}^n$ (Theorem \ref{Hasse}). 
In addition, we characterize algebras with radical square zero having only finitely many isomorphism classes of support $\tau$-tilting modules in terms of their valued quivers and Dynkin diagrams (Corollary \ref{tau-tilt-finite}). 
This is a generalization of the result of Adachi \cite[Theorem 3.1]{Ada1} (and also \cite{Zha}) to non-algebraically closed fields. 
For such algebras, we can calculate the number of support 
$\tau$-tilting modules since the number of tilting modules for 
hereditary algebras of Dynkin types is already known (see for example \cite{ONFR}). 
As an application, we determine the number of support $\tau$-tilting modules for special classes of symmetric algebras, called Brauer line algebras and Brauer odd-cycle algebras, 
which originates in modular representation theory of finite groups \cite{Alp, Ben}. The result is the following.

\begin{theorem}[Theorem \ref{line} and \ref{odd cycle}] \label{BLine}
The number of isomorphism classes of basic support $\tau$-tilting modules over 
a Brauer line algebra (resp. Brauer odd-cycle algebra) having $n$ edges is $\binom{2n}{n}$ (resp. $2^{2n-1}$).

\end{theorem}

In the forthcoming paper \cite{AA}, as an analog of Theorem \ref{BLine}, Adachi and the author compute the number of isomorphism classes of support $\tau$-tilting modules for an arbitrary symmetric algebra with radical cube zero over an algebraically closed field.

This paper is organized as follows. 
In Section \ref{preliminaries}, we recall a relationship between torsion classes, $\tau$-tilting modules and two-term silting complexes.
In Section \ref{hereditary algebras with RSZ}, we begin with introducing the sign-decomposition of $\tors A$, which is central in this paper, in a general setting. 
After that, we show a particular case of Theorem \ref{Theorem A}, when $A$ is a hereditary algebra with radical square zero.  
In Section \ref{algebras with RSZ}, we prove the main bijections in Theorem \ref{Theorem A} for an arbitrary algebra with radical square zero. 
This is achieved by reduction to the hereditary case in the previous section. Using this bijection, we compute $g$-vectors and the Hasse quiver of support $\tau$-tilting modules. 
In Section \ref{tau-tilting-finiteness}, we study algebras with radical square zero having only finitely many isomorphism classes of support $\tau$-tilting modules. As an application, we calculate the number of support $\tau$-tilting modules over Brauer line algebras and Brauer cycle algebras.

Throughout this paper, 
by an algebra we mean a finite dimensional basic algebra over a field $k$,  
by a module we mean a finitely generated right module.  
For a given algebra $A$, we denote by $\module A$ (resp. $\proj A$) the category of $A$-modules (resp. projective $A$-modules).  
An $A$-module $M$ is said to be {\it basic} if it is isomorphic to a direct sum of indecomposable modules which are mutually non-isomorphic.
Finally, for a given positive integer $m$, let $[m]:=\{1,\ldots, m\}$.

\ \\ \noindent
{\bf Acknowledgements.}
The author would like to express his deep gratitude to his supervisor Osamu Iyama 
for his support and advice. 
He would also like to thank Eric Darp{\"o}, Takahide Adachi and Aaron Chan for helpful comments and suggestions.

\section{Preliminaries} \label{preliminaries}
Throughout this section, let $A$ be a finite dimensional basic algebra.
We recall a relationship between torsion classes, support $\tau$-tilting modules and two-term silting complexes. 

\subsection{Torsion classes}
For a given $A$-module $M$, we denote by $\Fac M$ (respectively, $\Sub M$) 
the full subcategory of $\module A$ consisting of factor modules (respectively, submodules) of finite direct sums of copies of $M$. 

For a class $\mathcal{C}$ of $\module A$, we define its orthogonal categories by  
$$
\mathcal{C}^{\perp} := \{ X \in \module A \mid \Hom_{A} (\mathcal{C}, X)=0\}, \  
^{\perp}\mathcal{C} := \{ X \in \module A \mid \Hom_{A} (X, \mathcal{C})=0\}.   
$$
We say that a full subcategory $\mathcal{T}$ of $\module A$ is 
a {\it torsion class} (respectively, {\it torsion-free class})  
if it is closed under extensions, factor modules (respectively, submodules) and isomorphisms. 
Let $\tors A$ be the set of torsion classes of $\module A$. 
Then it forms a partially ordered set with respect to inclusion.
A torsion class $\mathcal{T}$ is said to be 
\begin{itemize}
\item {\it functorially finite} if it is of the form $\mathcal{T}=\Fac M$ for some $M \in \module A$, and 
\item {\it faithful} if it contains $DA$, 
\end{itemize}
where $D:=\Hom_k(-, k)$ denotes the $k$-dual.  
Let $\ftors A$ (respectively, $\fators A$) be the set of functorially finite (respectively, faithful) torsion classes of $\module A$.
Finally, let $\ffators A:=\ftors A \cap \fators A$. 
A {\it torsion pair} is a pair $(\mathcal{T}, \mathcal{F})$ consisting of a torsion class $\mathcal{T}$ and a torsion-free class $\mathcal{F}$ in $\module A$ satisfying 
$\mathcal{F}= \mathcal{T}^{\perp}$ and $\mathcal{T}={^\perp  \mathcal{F}}$. 
A torsion pair $(\mathcal{T}, \mathcal{F})$ is said to be {\it splitting} if each indecomposable $A$-module lies either in $\mathcal{T}$ or in $\mathcal{F}$.


\subsection{Support $\tau$-tilting modules} 
We denote by $\tau$ the Auslander-Reiten translation for $A$.
For a given $A$-module $M$, we denote the number of isomorphism classes of indecomposable direct summands of $M$ by $|M|$. 

\begin{definition}
\rm \label{def of tau-tilt}
Let $M$ be an $A$-module.  
\begin{enumerate}
\item $M$ is said to be {\it $\tau$-rigid} if $\Hom_{A}(M, \tau M)=0$. 
\item $M$ is said to be {\it $\tau$-tilting} if it is $\tau$-rigid and satisfies $|M|=|A|$. 
\item $M$ is said to be {\it support $\tau$-tilting} if there exists an idempotent $e$ of $A$ such that $M$ is a $\tau$-tilting $(A/\langle e \rangle)$-module.
\end{enumerate}
\end{definition}

From the next result, we find that support $\tau$-tilting modules is a generalization of tilting modules.
Here, we say that an $A$-module $M$ is {\it tilting} if it satisfies the following three conditions:
(i) $M$ is {\it rigid}, i.e., $\Ext^1_A(M,M)=0$, (ii) the projective dimension of $M$ is at most $1$, and (iii) $|M|=|A|$. It is known that any tilting $A$-module $M$ is {\it faithful}, that is, $DA\in \Fac M$. 

\begin{proposition} \rm \label{hereditary tilting} \cite[Proposition 2.2(b)]{AIR} The following hold.
\begin{enumerate}
\item Let $M$ be an $A$-module. If $M$ is $\tau$-rigid, then it is rigid. The converse holds if the projective dimension of $M$ is at most $1$.
\item  A tilting module is precisely a faithful support $\tau$-tilting module.
\end{enumerate}
\end{proposition}

We denote by $\stau A$ (resp. $\tilt A$) the set of isomorphism classes of basic support $\tau$-tilting (resp. tilting) $A$-modules. 
From the above proposition, we regard $\tilt A$ as a subset of $\stau A$.

\begin{proposition}\cite[Proposition 2.4 and Theorem 2.10]{AIR} \label{char-tau} 
The following hold. 
\begin{enumerate}
\item Let $M \in \module A$ and $P^{-1} \overset{f}{\rightarrow} P^0 \rightarrow M \rightarrow 0$  a minimal projective presentation of $M$. 
Then $M$ is $\tau$-rigid if and only if the map
$(f, M) \colon \Hom_A (P^0, M) \rightarrow \Hom_A(P^{-1}, M)
$ is surjective.
\item 
Any $\tau$-rigid $A$-module is a direct summand of a $\tau$-tilting $A$-module. 
\end{enumerate}
\end{proposition}

A support $\tau$-tilting module can be regarded as a certain pair of modules. 
Now, let $\add M$ be the smallest full subcategory of $\module A$ that contains an $A$-module $M$ and is closed under direct sums and direct summands.

\begin{definition} \rm 
Let $M \in\module A$ and $Q \in \proj A$. 
\begin{enumerate} 
\item A pair $(M, Q)$ is called {\it $\tau$-rigid pair} if $M$ is $\tau$-rigid and 
$\Hom_A(Q,M)=0$. 
\item $(M,Q)$ is called {\it $\tau$-tilting pair} 
(respectively, {\it almost complete $\tau$-tilting pair}) 
if it is a $\tau$-tilting pair and satisfies $|M|+|Q|=|A|$
(respectively, $|M|+|Q|=|A|-1$). 
\item We say that a pair $(M,Q)$ is {\it basic} if both $M, P$ are.
\end{enumerate}
\end{definition}

\begin{proposition} \cite[Proposition 2.3]{AIR}
For any support $\tau$-tilting module $M$, there exists $Q \in \proj A$ such that $(M,Q)$ is a $\tau$-tilting pair for $A$. 
In addition, if $(M,Q)$ and $(M,Q')$ are $\tau$-tilting pairs for $A$, 
then $\add Q = \add Q'$ holds.
\end{proposition}

Therefore, up to isomorphism, every basic support $\tau$-tilting module $M$ determines 
a unique projective module $Q$ such that $(M, Q)$ is a basic $\tau$-tilting pair for $A$.
In this sense, we identify a basic support $\tau$-tilting $A$-module $M$ 
with a $\tau$-tilting pair $(M,Q)$ for $A$. 
In this case, we write $(M,Q) \in \stau A$ if there are no confusion.

\subsection{Two-term silting complexes}
Let $\mathsf{C}^{\rm b}(\proj A)$ be the category of bounded complexes of $\proj A$ and 
$\Kb(\proj A)$ its homotopy category.

\begin{definition} \rm 
Let $T=(T^i, d^i)$ be a complex in $\Kb(\proj A)$. 
\begin{enumerate}
\item $T$ is said to be {\it two-term} if $T^i=0$ for all integer $i\neq -1,0$.
\item $T$ is said to be {\it presilting} if $\Hom_{\Kb(\proj A)}(T,T[i])=0$ for any integer $i>0$. 
\item $T$ is said to be {\it silting} if it is presilting and $\thick T = \Kb(\proj A)$, where 
$\thick T$ is the smallest full subcategory of $\Kb(\proj A)$ that contains $T$ and is closed under 
shifts, direct summands, mapping cones and isomorphisms.
\end{enumerate}
\end{definition}

We denote by $\twosilt A$ the set of isomorphism classes of basic two-term silting complexes for $A$. 
The following are basic properties of two-term presilting complexes. 
For a complex $T \in \Kb(\proj A)$, we denote by $|T|$ 
the number of isomorphism classes of indecomposable direct summands of $T$.

\begin{proposition}
\label{properties of two-silt}
Let $T$ be a basic two-term presilting complex in $\Kb(\proj A)$. Then the following hold. 
\begin{enumerate}
\item \cite[Proposition 2.16]{Ai} $T$ is a direct summand of a two-term silting complex for $A$. 
In particular, $T$ is silting if and only if $|T|=|A|$. 
\item \cite[Lemma 2.25]{AI} $\add T^0 \cap \add T^{-1}=0$ holds. 
\item \cite[Theorem 2.27]{AI} Assume that $T$ is silting.
Then the set of isomorphism classes of indecomposable direct summands of $T$ forms a basis of the the Grothendieck group of the triangulated category $\Kb(\proj A)$.  
\item If $T$ is silting, then each indecomposable projective $A$-module lies either in $\add T^0$ or in $\add T^{-1}$. 
\end{enumerate}
\end{proposition} 

\begin{proof}
(4) It is immediate from (2) and (3).
\end{proof}

\begin{lemma} \label{exa}
Let $T = (T^{-1} \overset{d_T}{\rightarrow} T^0)$ and $U=(U^{-1} \overset{d_U}{\rightarrow} U^0)$ be two-term complexes in $\Kb(\proj A)$. 
Then there is an exact sequence: 
\begin{eqnarray*}
0 \longrightarrow \Hom_{\mathsf{C}^{\rm b}(\proj A)}(T,U) &\overset{\iota}{\longrightarrow}& \Hom_{A}(T^{-1}, U^{-1})
\oplus \Hom_{A}(T^{0}, U^{0}) \\ &\overset{\eta}{\longrightarrow}& \Hom_{A}(T^{-1},U^{0})
\overset{\pi}{\longrightarrow}
\Hom_{\Kb(\proj A)}(T,U[1]) \longrightarrow  0 
\end{eqnarray*}
where $\iota, \pi$ are natural injection and surjection respectively,  
and $\eta$ is defined by $\eta(h^{-1}, h^0) := h^0d_T - d_Uh^{-1}$ for all
$h^{-1}\in \Hom_{A}(T^{-1}, U^{-1})$ and $h^{0}\in \Hom_{A}(T^{0}, U^{0})$.
\end{lemma}

\begin{proof}

Let $h^{-1}\in \Hom_{A}(T^{-1}, U^{-1})$ and $h^{0}\in \Hom_{A}(T^{0}, U^{0})$.
Then $(h^{-1}, h^{0}) \in \Hom_{\mathsf{C}^{\rm b}(\proj A)}(T,U)$ if and only if $h^0d_T= d_Uh^{-1}$, if and only if $\eta(h^{-1}, h^{0})=0$. 
Therefore, $\Image \iota = \Ker \eta$ holds.

On the other hand, let $f\in \Hom_A(T^{-1}, U^0)$. 
Then $\pi(f) \in \Hom_{\Kb(\proj A)}(T,U[1])$ is null-homotopic if and only if 
there exist $h^{-1}\in \Hom_{A}(T^{-1}, U^{-1})$ and $h^{0}\in \Hom_{A}(T^{0}, U^{0})$ such that
$f=h^0d_T - d_Uh^0$. 
This is equivalent to the condition that $f \in \Image \eta$.
Therefore, $\Image \eta = \Ker \pi$. 
Thus, we have the assertion.
\end{proof}

As a summary of this subsection, we give the following result.

\begin{theorem} \cite[Theorem 0.5]{AIR} \label{bijection}
There is bijections between $\ftors A$, $\stau A$ and $\twosilt A$.
\end{theorem}

Here, we explain bijections in the statement. 
The map from $\stau A$ to $\ftors A$ is given 
by the correspondence $(M,Q) \mapsto \Fac M$ for any $\tau$-tilting pair $(M,Q)$ for $A$.
In addition, it induces a bijection $\tilt A \rightarrow \ffators A$ (\cite[Corollary 2.8]{AIR}). 
On the other hand, the bijection $\stau A \rightarrow \twosilt A$ is given by the correspondence $(M,Q) \mapsto T_{(M,Q)}:= (Q \oplus P^{-1} \overset{(0\ f)}{\rightarrow} P^0)$ for any $\tau$-tilting pair $(M,Q)$ for $A$, where $P^{-1} \overset{f}{\rightarrow} P^0\rightarrow M \rightarrow 0$ is a minimal projective presentation of $M$. 
Conversely, the inverse map is given by taking $0$-th cohomology $H^0(T)$ of a given two-term complex $T$.

From these bijections, $\stau A$ and $\twosilt A$ have a structure of partially ordered sets induced by $\ftors A$.

Next, we recall the definition of $g$-vectors. 
Let $\mathcal{S}(A)=\{e_1, \ldots, e_n\}$ be a complete set of primitive orthogonal idempotents of $A$. 
Let $P(i):=e_i A$ and $S(i):=P(i)/\rad P(i)$ be projective and simple $A$-modules for all $i \in \{1,\ldots,n\}$ respectively. 

\begin{definition}\rm
Let $T=(T^{-1} \overset{d_T}{\rightarrow} T^0)$ be a two-term complex in $\Kb(\proj A)$. 
Suppose that $T^{0}= \bigoplus_{i=1}^n P(i)^{m_i}$ and $T^{-1}=\bigoplus_{i=1}^n P(i)^{m_i'}$.
The {\it $g$-vector} $g_A^T$ of $T$ is defined by
$$g_A^T:= {^t(m_1-m_1', m_2-m_2', \ldots, m_n-m_n')} \in \mathbb{Z}^n.$$
We write $g^T$ for the $g$-vector of $T$ if the considering algebra is clear.
\end{definition}

In addition, the {\it g-vector} $g^M$ of an $A$-module $M$ is defined by 
that of a two-term complex $(P^{-1} \overset{f}{\rightarrow} P^0)$, where $P^{-1} \overset{f}{\rightarrow} P^0 \rightarrow M \rightarrow 0$ is a minimal projective presentation of $M$. 
If $(M,Q)$ is a $\tau$-tilting pair for $A$, then we have the equality $g^{(M,Q)}:= g^M-g^Q = g^{T_{(M,Q)}}$, where $T_{(M,Q)}$ is a two-term silting complex corresponding to $(M,Q)$.

We need the following properties of $g$-vectors.

\begin{proposition} \label{properties of g-vectors}
\begin{enumerate}
\item \cite[Corollary 6.7(a)]{DIJ} 
The map $\stau A \rightarrow \mathbb{Z}^n$ mapping 
$(M,Q) \mapsto g^{(M,Q)}$ is injective. 
\item If $(M,Q)$ is a $\tau$-tilting pair for $A$, 
then $g_i^{(M,Q)}\neq 0$ for all $i =1,\ldots,n$. 
\end{enumerate}
\end{proposition}

\begin{proof}
(2) If $(M,Q)$ is a $\tau$-tilting pair for $A$, then we have 
$g^{(M,Q)}=g^{T_{(M,Q)}}$ where $T_{(M,Q)}$ is a two-term silting complex corresponding to $(M,Q)$. 
So the assertion follows from Proposition \ref{properties of two-silt}(4).
\end{proof}

\subsection{Valued quivers} \label{valued quiver}
We end this section with introduction of the notion of valued quivers of finite dimensional algebras. 

Let $A$ be an algebra and $\mathcal{S}(A)=\{e_1, \ldots, e_n\}$ a complete set of primitive orthogonal idempotents of $A$. 
For a given $A$-module $M$, we denote by $c_A^M$ the {\it dimension vector} of $M$, 
namely, $i$-th entry of $c_A^M$ is equal to the multiplicity of $S(i)$ in its composition series. 
The matrix $\mathrm{C}_A=[c_A^{P(1)} \cdots c_A^{P(n)}]$ is called {\it Cartan matrix} of $A$. 

\begin{definition} \cite{DR} \rm  \label{valued}
Suppose that $A/\rad A=\prod_{i=1}^n F_i$, where $F_i$ is a division algebra for each $i$.
Each pair $(i,j)$ determines an $F_i$-$F_j$-bimodule ${_i}N_j:= e_i(\rad A/ \rad^2 A)e_j$. 
We define a {\it valued quiver} $\Gamma_A$ of $A$ as follows: 
The set of vertices of $\Gamma_A$ is $[n]=\{1,\ldots, n\}$, and 
there is an arrow from $i$ to $j$ if ${_i}N_j \neq 0$ and 
assigning to this arrow the pair of non-negative integers 
$$(d_{ij}', d_{ij}'') = ({\rm dim} ({_i}N_j)_{F_j}, {\rm dim}_{F_i}({_i}N_j) ).$$
Note that the number $d_{ij}'$ is equal to the multiplicity of $S(j)$ in $\rad P(i)/\rad^2 P(i)$ as direct summands. 
\end{definition}

\begin{remark}\rm 
A finite quiver in usual sense can be regarded as a valued quiver by replacing $m$ arrows from $i$ to $j$ by one arrow with valuation $(m,m)$, vice versa.
\end{remark}

Now, we consider a hereditary algebra $A$.
In this case, a valued quiver $\Gamma_A$ of $A$ has no oriented cycles (see \cite[Section III.1, Lemma 1.12]{ARS}).
Due to the results of \cite[Section 2]{DR}, we can consider a simultaneously reflection 
at sinks as follows: 
Let $\mathbf{a}$ be a set of sinks of $\Gamma_A$ which are not sources. 
For each vertex $a \in \mathbf{a}$, we define a $\mathbb{Z}$-linear transformation 
$s_a \colon \mathbb{Z}^n \rightarrow \mathbb{Z}^n$, called {\it reflection} at $a$, 
by $s_a x = y$ where $x_i =y_i$ for $i\neq a$ and 
$$
y_a = -x_a + \sum_{i \rightarrow a} d_{ia}' x_i. 
$$
Let $\mathrm{S}_{\mathbf{a}}:= \prod_{a\in \mathbf{a}} s_a$. 
On the other hand, let
$$
T[\mathbf{a}] := \tau^- S_A(\mathbf{a}) \oplus \big(\bigoplus_{i \notin \mathbf{a}} P_A(i)\big)
\quad where \quad S_A(\mathbf{a}):=\bigoplus_{a \in \mathbf{a}} S_A(a),$$ 
then $T[\mathbf{a}]$ is a tilting $A$-module. 
Let $B:=\End_A(T[\mathbf{a}])$ and $\mathcal{S}(B):=\{e'_1, \ldots, e'_n\}$ a complete set of primitive orthogonal idempotents of $B$ induced by $\mathcal{S}(A)$.

\begin{proposition} \label{reflection}
In the above, the following hold.

\begin{enumerate}
\item[(a)] $B$ is a hereditary algebra, and its valued quiver $\Gamma_B$ is obtained by reversing the direction of all arrows of $\Gamma_A$ incident to $a \in \mathbf{a}$.
\item[(b)] We have $\Fac T[\mathbf{a}] =\Fac \big(\bigoplus_{i \notin \mathbf{a}} P(i)\big) = {^\perp S_A(\mathbf{a})}$ in $\module A$. 

\item [(c)]
A pair $({^\perp S_A(\mathbf{a})}, \add S_A(\mathbf{a}))$ is a splitting torsion pair in $\module A$. Similarly,  $(\add S_B(\mathbf{a}), S_B(\mathbf{a})^{\perp})$ is a splitting torsion pair in $\module B$. 
Moreover, there are equivalences of exact categories: 
$$
^\perp S_A(\mathbf{a})\ 
\begin{xy}
(0,0)="A", (13,0)="B",
\ar @<1mm> "A";"B"^{\mathbf{R}_{\mathbf{a}}}
\ar @<1mm> "B";"A"^{\mathbf{L}_{\mathbf{a}}}
\end{xy} 
\ S_B(\mathbf{a})^\perp \quad  \text{and} \quad
\add S_A(\mathbf{a}) \ 
\begin{xy}
(0,0)="A", (8,0)="B",
\ar  "A";"B"^{(-)_B}
\end{xy} 
\ \add S_B(\mathbf{a}).
$$
where $\mathbf{R}_{\mathbf{a}}:=\Hom_{A}(T[\mathbf{a}],-)$ and $\mathbf{L}_{\mathbf{a}}:=-\otimes_{B}T[\mathbf{a}]$, and there is a natural isomorphism $\Ext^1_A(T[\mathbf{a}], S_A(\mathbf{a})) \cong S_B(\mathbf{a})$.
\item[(d)] If $M \in {^\perp S_A(\mathbf{a})}$, then $c_B^{\mathbf{R}_{\mathbf{a}}(M)} = \mathrm{S}_{\mathbf{a}} \cdot c_A^{M}$. 
\end{enumerate}
\end{proposition}

\begin{proof}
We get assertions by applying \cite[Theorem 2.1]{DR} iteratively.
\end{proof}

The following characterization is an analog of \cite{Ga}.

\begin{proposition}  \label{tilting number} \cite[Theorem(a)]{DR}
Let $A$ be a finite dimensional connected hereditary algebra and $\Gamma_A$ a valued quiver of $A$. 
Then the following conditions are equivalent. 
\begin{enumerate}
\item $A$ is representation finite. 
\item $\tilt A$ is finite.
\item The underlying graph of $\Gamma_A$ is one of Dynkin diagrams
$\mathbb{A}_n$, $\mathbb{B}_n$, $\mathbb{C}_n$, $\mathbb{D}_n$, 
$\mathbb{E}_6$, $\mathbb{E}_7$, $\mathbb{E}_8$,
$\mathbb{F}_4$ and $\mathbb{G}_2$.
\end{enumerate}
\end{proposition}

In the case of Proposition \ref{tilting  number}, the number of tilting modules is given by the following table, where $C_n:=\frac{1}{n+1}\binom{2n}{n}$ is the $n$-th Catalan number (see for example  \cite{ONFR}):

{\rm
\begin{table}[h]
  \begin{tabular}{|c||c|c|c|c|c|c|c|c|c|}
\hline
    $\Gamma_A$ & $\mathbb{A}_n$ & $\mathbb{B}_n, \mathbb{C}_n$ & $\mathbb{D}_n$ & 
$\mathbb{E}_6$ & $\mathbb{E}_7$ & $\mathbb{E}_8$ &
$\mathbb{F}_4$ & $\mathbb{G}_2$ \\ \hline
 $\# \tilt A$  & $C_n$& $\binom{2n-1}{n-1}$& $\frac{3n-4}{2n-2}\binom{2n-2}{n-2}$& 418 & 2431 & 17342 & 66& 5  \\ \hline
  \end{tabular}
\end{table}
}

\section{Hereditary algebras with radical square zero} \label{hereditary algebras with RSZ}

We begin this section with introducing the sign-decomposition of the set of torsion classes of a given finite dimensional basic algebra. 
After that, we study the set of torsion classes of hereditary algebras with radical square zero via this decomposition. 
Now, we naturally identify a map $\epsilon \colon [n] \rightarrow \{\pm1\}$ with $n$ tuple $(\epsilon(1),\ldots, \epsilon(n)) \in \{\pm1\}^n$.

\subsection{Sign-decomposition}
Let $A$ be a finite dimensional algebra and $\mathcal{S}(A)=\{e_1,\ldots, e_n\}$ a complete set of primitive orthogonal idempotents of $A$. 
Let $\epsilon \colon [n] \rightarrow \{\pm1\}$ be a map. 
For each $\sigma \in \{\pm\}$, let $e^{\epsilon, \sigma}:=\sum_{i \in \epsilon^{-1}(\sigma)} e_i\in A$, and let $P_A^{\epsilon, \sigma} := e^{\epsilon, \sigma} A$ and $S_A^{\epsilon, \sigma} := P_A^{\epsilon, \sigma}/ \rad P_A^{\epsilon, \sigma}$ be projective and semisimple $A$-modules respectively. Similarly, let $I_A^{\epsilon, \sigma}$ be an injective $A$-module corresponding to $e^{\epsilon,\sigma}$.
Note that $P_A^{\epsilon, \sigma}$ (resp. $S_A^{\epsilon, \sigma}$) is isomorphic to a direct sum of indecomposable projective modules $P(i)$ (resp. simple modules $S(i)$) for all $i \in \epsilon^{-1}(\sigma)$.

For a given map $\epsilon \colon [n] \rightarrow \{\pm1\}$, let 
$$
\tors_{\epsilon} A := \{\mathcal{T} \in \tors A \mid S(i) \in \mathcal{T} \Leftrightarrow \epsilon(i)=1\}. 
$$
Clearly, $\tors A$ is decomposed into these $2^n$ subsets, which we call sign-decomposition of $\tors A$. 
From this decomposition, we can assign a map $\epsilon$ to each torsion class of $\module A$.
On the other hand, it is easy to see that $\tors_{\epsilon} A$ coincides with an interval 
$[\Filt S_A^{\epsilon, +}, \Fac P_A^{\epsilon, +}]_{\subseteq}$ in the partially ordered set $\tors A$ for every $\epsilon$, where 
$\Filt S_A^{\epsilon, +}$ is the full subcategory of $\module A$ consisting of all $A$-modules filtered by $S_A^{\epsilon, +}$.

First observation is that our signature is compatible with taking factor algebras.
Let $B$ be a factor algebra of $A$. 
Assume that the induced map $\mathcal{S}(A) \rightarrow \mathcal{S}(B)$ is a bijection. 
In this situation, we always consider the set $\tors_{\epsilon} B$ with respect to 
the index $e'_1,\ldots, e_n'$, where $e'_i$ is an element in  $\mathcal{S}(B)$ corresponding to $e_i$ for every $i \in [n]$.
On the other hand, we can naturally regard $\module B$ as a full subcategory of $\module A$. For a given full subcategory $\mathcal{T}$ of $\module A$, let $\overline{\mathcal{T}}:=\mathcal{T} \cap \module B$. 

\begin{proposition} \label{sign and factor}
Let $B$ be a factor algebra of $A$ such that the induced map $\mathcal{S}(A) \rightarrow \mathcal{S}(B)$ is a bijection. 
Then we have a commutative diagram
$$
\begin{xy}
(0,0)*+{\tors A}="1", (50,0)*+{\tors B}="2", 
(25, -13)*+{\{ \pm1\}^n}="3", 
(25, -5)*={\circlearrowright},
{"1" \ar^{\overline{(-)}} "2"},
{"2" \ar^{\rm sign} "3"},  
{"1" \ar_{\rm sign} "3"}, 
\end{xy}
$$
\end{proposition}

\begin{proof}
By \cite[Proposition 5.7]{DIRRT}, we have the map $\overline{(-)} \colon \tors A \rightarrow \tors B$. 
Then the assertion follows from the fact that, for a given torsion class $\mathcal{T}$ of $\module A$, $S(i) \in \mathcal{T}$ if and only if $S(i) \in \overline{\mathcal{T}}$ for any $i \in [n]$. 
\end{proof}

For each map $\epsilon \colon [n] \rightarrow \{\pm1\}$, let 
\begin{itemize}
\item $\ftors_{\epsilon} A := \tors_{\epsilon} A \cap \ftors A$, 
\item $\stau_{\epsilon} A := \{ (M,Q) \in \stau A \mid \text{
$g^{(M, Q)}_i \in \epsilon(i) \cdot \mathbb{Z}_{>0}$ for all $i \in [n]$}\}$,
\item $\twosilt_{\epsilon} A :=\{T \in \twosilt A \mid 
\text{$g^T_i \in \epsilon(i) \cdot \mathbb{Z}_{>0}$ for all $i\in[n]$}\}$.
\end{itemize}

By Proposition \ref{properties of g-vectors}(2), $\ftors A$ (resp. $\stau A$, $\twosilt A$) is a disjoint union of these $2^n$ subsets respectively. 
The following result is basic but important. 

\begin{proposition} \label{restriction-epsilon}
For each map $\epsilon \colon [n] \rightarrow \{\pm1\}$, bijections in Theorem \ref{bijection}  give isomorphisms of partially ordered sets among 
$\ftors_{\epsilon} A$, $\stau_{\epsilon} A$ and $\twosilt_{\epsilon} A$.
\end{proposition}

\begin{proof}
Let $\epsilon \colon [n] \rightarrow \{\pm1\}$ be a map. 
By the definition of $g$-vectors, we obtain a bijection between
$\stau_{\epsilon} A$ and $\twosilt_{\epsilon} A$. 

Next, we show that $\Fac M \in \ftors_{\epsilon} A$ for any $\tau$-tilting pair $(M,Q)$ in $\stau_{\epsilon} A$. 
Let $(M,P)$ be a $\tau$-tilting pair for $A$ and $P^{-1} \rightarrow P^0 \rightarrow M \rightarrow 0$ a minimal projective presentation of $M$. 
If $(M,Q) \in \stau_{\epsilon} A$, then $\add P^0 = \add P_A^{\epsilon,+}$ and $\add (Q\oplus P^{-1}) = \add P_A^{\epsilon,-}$ holds. 
So, we have $S_A^{\epsilon,+} \in \Fac M$ and $M \in \Fac P_A^{\epsilon,+}$. 
It implies that $\Fac M \in \ftors_{\epsilon}A$. 

Since $\stau A$ is decomposed into $\stau_{\epsilon} A$ for all $\epsilon \colon [n] \rightarrow \{\pm1\}$, we complete the proof.
\end{proof}

In considering two-term complexes with signatures, 
the following notations are useful: For a given map 
$\epsilon \colon [n] \rightarrow \{\pm1\}$, let $\Ktwo(\proj A)$ be the full subcategory of 
$\Kb(\proj A)$ consisting of complexes isomorphic to a two-term complex  
$T=(T^{-1}\rightarrow T^0)$ such that 
$T^{0} \in \add P_A^{\epsilon, +}$ and $T^{-1} \in \add P_A^{\epsilon, -}$, 
in other words, $g^T \in \{x \in \mathbb{Z}^n\mid 
\text{$x_i \in \epsilon(i)\cdot \mathbb{Z}_{\geq 0}$ for all $i\in [n]$}\}$.

\begin{proposition}
For each map $\epsilon \colon [n] \rightarrow \{\pm1\}$, the following hold.
\begin{enumerate}
\item Let $V\rightarrow U \rightarrow T \rightarrow V[1]$ be a triangle in $\Kb(\proj A)$. If $T, V \in \Ktwo(\proj A)$, then $U \in \Ktwo(\proj A)$. 
\item Any basic complex $T \in \Ktwo(\proj A)$ satisfies $\add T^{0} \cap \add T^{-1}=0$. 
\item $\twosilt_{\epsilon} A = \twosilt A \cap \Ktwo(\proj A)$.
\end{enumerate}
\end{proposition}

\begin{proof}
(1) Let $V \rightarrow U \rightarrow T \rightarrow V[1]$ be a triangle in $\Kb(\proj A)$ with two-term complexes $T=(T^{-1} \overset{d_T}{\rightarrow} T^{0})$ and $V = (V^{-1} \overset{d_V}{\rightarrow} U^{0})$. 
Up to isomorphism, we can assume that both differential $d_T, d_V$ are in the 
radical of the category $\proj A$.
Since $U$ is given by the mapping cone of $T[-1] \rightarrow V$, 
the underlying modules $U^0$ and $U^{-1}$ are isomorphic to 
$T^{0} \oplus V^{0}$ and $T^{-1}\oplus V^{-1}$ respectively. 
So, $T,V \in \Ktwo(\proj A)$ implies $U\in \Ktwo(\proj A)$.
The assertion (2) and (3) is clear. 
\end{proof}

\subsection{Hereditary algebras with radical square zero} \label{torsion class for hereditary}
In the rest of this section, we consider a finite dimensional hereditary algebra with radical square zero. Here, we say that $A$ is an algebra with radical square zero if the square of the Jacobson radical $J_A$ of $A$ is $0$.

Let $A$ be a connected hereditary algebra with radical square zero and $\mathcal{S}(A)=\{e_1,\ldots, e_n\}$ a complete set of primitive orthogonal idempotents of $A$.  
If $A$ is non-simple, then the valued quiver $\Gamma_A$ of $A$ is connected and {\it bipartite}, that is, every vertex of $\Gamma_A$ is either a sink or a source. 
In fact, if there exists a path $i \rightarrow j \rightarrow k$ of length $2$ in $\Gamma_A$, 
it induces a monomorphism from $P(k)$ to $P(i) $ whose image is contained in the square of the radical of $P(i)$; it contradicts to the assumption that $A$ is an algebra with radical square zero.
So, we can define a map $\epsilon_A \colon [n] \rightarrow \{\pm1\}$ by $\epsilon_A(i)=1$ if $i$ is a source and $\epsilon_A(i)=-1$ if $i$ is a sink. 
To the map $\epsilon_A$, we associate a semisimple algebra $e^{\epsilon_A,\sigma} (A/J_A) e^{\epsilon_A,\sigma}$ for each $\sigma \in \{\pm\}$.
The Jacobson radical $J_A$ of $A$ has a natural structure of  
$(e^{\epsilon_A,+} (A/J_A) e^{\epsilon_A,+})$-$(e^{\epsilon_A, -} (A/J_A) e^{\epsilon_A, -})$-bimodule. 
Since we have 
$e^{\epsilon_A, -} (A/J_A) e^{\epsilon_A, +}=0$, $A$ is isomorphic to the following upper triangular matrix algebra (see \cite[Section III.2]{ARS} for detail):
\begin{equation} \label{A as upper matrix algebra}
A \cong \begin{pmatrix} 
e^{\epsilon_A,+} (A/J_A) e^{\epsilon_A,+}  & J_A \\ 
0 & e^{\epsilon_A, -} (A/J_A) e^{\epsilon_A, -}
\end{pmatrix}.
\end{equation}

From the above description, we find that each simple module is either projective or injective. 

\begin{lemma} \label{simples}
$S_A^{\epsilon_A,-}$ (resp. $S_A^{\epsilon_A,+}$) is a semisimple projective (resp. injective) $A$-module that is isomorphic 
to $P_A^{\epsilon_A, -}$ (resp. $I_{A}^{\epsilon_A,+}$). 
\end{lemma}

\begin{proof}
It is just \cite[Section III.2, Proposition 2.5(b) and (c)]{ARS}.
\end{proof}

Next, we consider a simultaneously reflection at all sinks of $\Gamma_A$.  
Note that the set $\epsilon_A^{-1}(-1)$ coincides with the set of sinks in $\Gamma_A$.
So, the associated tilting module to $\epsilon_A^{-1}(-1)$ is given by 
$T[\epsilon_A^{-1}(-1)]= \tau^{-1}S_A^{\epsilon_A, -} \oplus P_A^{\epsilon_A, +}$.
By Proposition \ref{reflection}, 
$\End_A(T[\epsilon_A^{-1}(-1)])$ is again a hereditary algebra with radical square zero, 
and its valued quiver is obtained by reversing all arrows of $\Gamma_A$. 
From the similar reason as (\ref{A as upper matrix algebra}), we have the following expression as an upper triangular matrix algebra:
\begin{equation} \label{A! as upper matrix algebra}
\End_{A}(T[\epsilon_A^{-1}(-1)]) \cong 
\begin{pmatrix} 
e^{\epsilon_A, -} (A/J_A) e^{\epsilon_A, -}  & DJ_A \\ 
0 & e^{\epsilon_A, +} (A/J_A) e^{\epsilon_A, +}.
\end{pmatrix}
\end{equation}
We write this upper triangular matrix algebra by $A^!$. 
Let $\mathcal{S}(A^!):=\{e'_1,\ldots, e'_n\}$ be a complete set of primitive orthogonal idempotents of $A^!$ corresponding to $\mathcal{S}(A)$. 

In the above, let 
$(\mathcal{T}_{\epsilon_A}, \mathcal{F}_{\epsilon_A}) := 
({^\perp (S_A^{\epsilon_A, -})}, \add S_A^{\epsilon_A, -})$ and 
$(\mathcal{Y}_{\epsilon_A}, \mathcal{X}_{\epsilon_A}) :=(\add S_{A^!}^{\epsilon_A,-}, (S^{\epsilon_A, -}_{A^!})^{\perp})$ 
be torsion pairs in $\module A$ and in $\module A^!$ associated to $T[\epsilon_A^{-1}(-1)]$ respectively. 
By Proposition \ref{reflection}, there are equivalences of exact categories:
$$
\mathcal{T}_{\epsilon_A} \ 
\begin{xy}
{(0,0) \ar@<1mm>^{\mathbf{R}} (16,0)}, 
{(16,0) \ar@<1mm>^{\mathbf{L}} (0,0)}, 
\end{xy}
\ \mathcal{X}_{\epsilon_A} \quad and \quad 
\mathcal{F}_{\epsilon_A} \
\begin{xy}
{(0,0) \ar^{(-)_{A^!}} (12,0)}, 
\end{xy} \ 
\mathcal{Y}_{\epsilon_A}
$$
where $\mathbf{R}:=\mathbf{R}_{\epsilon_A^{-1}(-1)}, \mathbf{L}:=\mathbf{L}_{\epsilon_A^{-1}(-1)}$. 


On the other hand, if $A$ is simple, then $\Gamma_A$ consists of a vertex $\{1\}$ and no arrows. Let $\epsilon_A$ be one of maps from $\{1\}$ to $\{\pm1\}$. 
For both cases, let $A^!:= \End_A (A)$ and let $\mathbf{R}:= \Hom_A(A,-), \mathbf{L}:=-\otimes_A A$.
If $\epsilon(1)=1$, the associated torsion pairs are defined by $(\mathcal{T}_{\epsilon_A}, \mathcal{F}_{\epsilon_A}) :=  (\module A, 0)$ and $(\mathcal{Y}_{\epsilon_A}, \mathcal{X}_{\epsilon_A}) :=  (0, \module A^!)$. 
On the other hand, if $\epsilon(1)=-1$, then $(\mathcal{T}_{\epsilon_A}, \mathcal{F}_{\epsilon_A}) :=  (0, \module A)$ and $(\mathcal{Y}_{\epsilon_A}, \mathcal{X}_{\epsilon_A}) :=  (\module A^!, 0)$.

\begin{lemma} \label{simple-injective}
Assume that $A$ is non-simple.  Then the following hold.
\begin{enumerate}
\item $\mathcal{T}_{\epsilon_A} = \Fac P_A^{\epsilon_A,+}$ holds. 
In particular, if $\mathcal{T}\in \tors_{\epsilon_A} A$, then $\mathcal{T} \subseteq \mathcal{T}_{\epsilon_A}$.
\item If $M \in \mathcal{T}_{\epsilon_A}$, then $\rad M \in \mathcal{F}_{\epsilon_A}$.
\item $S_A^{\epsilon_A, +} \in \mathcal{T}_{\epsilon_A}$ and 
$\mathbf{R}(S_A^{\epsilon_A, +}) \cong I_{A^!}^{\epsilon_A,+}$. 
\item $S_{A^!}^{\epsilon_A, -}$ (resp. $S_{A^!}^{\epsilon_A, +}$) is a semisimple injective (resp. projective) $A^!$-module that is isomorphic to $I_{A^!}^{\epsilon_A,-}$ (resp. $P_{A^!}^{\epsilon_A, +}$). 
\end{enumerate}
\end{lemma}

\begin{proof}
(1): By Proposition \ref{reflection}(b), we have 
$\mathcal{T}_{\epsilon_A} := {^{\perp} (S_A^{\epsilon_A,+})} = \Fac P^{\epsilon_A, +}_A$.
Moreover, if $\mathcal{T}\in \tors_{\epsilon_A} A$, then $\mathcal{T} \subseteq \Fac P^{\epsilon_A, +}_A = \mathcal{T}_{\epsilon_A}$.

(2): Clearly, $\rad A \in \add S_A^{\epsilon_A,-} =\mathcal{F}_{\epsilon_A}$. 
For any $M\in \module A$, there is a surjection $A^{r} \rightarrow M$ for some integer $r$. It induces a surjection $(\rad A)^{r} \rightarrow \rad M$. 
Then $\rad A \in \mathcal{F}_{\epsilon_A}$ implies $\rad M \in \mathcal{F}_{\epsilon_A}$.

(3): It follows from \cite[Section VI, Lemma 4.9]{ASS}. 

(4): From the same reason as Lemma \ref{simples}, we have the assertion for $A^!$.
\end{proof}

The following is a basic result in this paper. 

\begin{theorem} \label{hereditary case}
Let $A$ be a connected hereditary algebra with radical square zero. 
Then there is an isomorphism of partially ordered sets
$$
R\colon \tors_{\epsilon_A} A\longrightarrow \fators A^!,
$$
where $R(\mathcal{T}):= \add (\mathbf{R}(\mathcal{T}), \mathcal{Y}_{\epsilon_A})$ for any $\mathcal{T} \in \tors_{\epsilon_A} A$. The inverse correspondence is given by $L \colon \mathcal{U} \mapsto \mathbf{L}(\mathcal{U}\cap \mathcal{X}_{\epsilon_A})$ for any $\mathcal{U} \in \fators A^!$.
\end{theorem}

\begin{proof}
Assume that $A$ is simple. 
If $\epsilon_A(1)=1$, then $\tors_{\epsilon_A} A =\{\module A\}$ and $\fators A^!=\{\module A^!\}$ hold. 
By definition, we have 
$R(\module A) = \add (\mathbf{R}(\module A), 0) = \module A^!$ and $L(\module A^!) = \mathbf{L}(\module A^!)= \module A$.    
Similarly, if $\epsilon(1)=-1$, then $\tors_{\epsilon_A} A = \{0\}$ and $\fators A^! = \{\module A^!\}$. So, we have $R(0)=\add (\mathbf{R}(0), \module A^!) = \module A^!$ and 
$L(\module A^!) = \mathbf{L}(\module A^! \cap 0) =0$.
Therefore, $R$ is a bijection when $A$ is simple.

Next, we assume that $A$ is non-simple. 
Firstly, we show that $R(\mathcal{T})$ is a faithful torsion class for any $\mathcal{T} \in \tors_{\epsilon_A} A$. 
Let $\mathcal{T} \in \tors_{\epsilon_A} A$. 
By Lemma \ref{simple-injective}(1), we have $\mathcal{T} \subset \mathcal{T}_{\epsilon_A}$ and $\mathbf{R}(\mathcal{T}) \subset \mathcal{X}_{\epsilon}$. 
Firstly, we check that $R(\mathcal{T})$ is closed under factor modules, namely, 
$\Fac R(\mathcal{T}) \subseteq R(\mathcal{T})$. 
We have
$$
\Fac R(\mathcal{T}) = \Fac (\mathbf{R}(\mathcal{T}), \mathcal{Y}_{\epsilon_A}) = \add (\Fac \mathbf{R}(\mathcal{T}), \mathcal{Y}_{\epsilon_A}), 
$$
where we use that fact that $\Fac \mathcal{Y}_{\epsilon_A} = \add \mathcal{Y}_{\epsilon_A}$. 
Since $\mathcal{Y}_{\epsilon_A} \subset R(\mathcal{T})$ holds by definition, 
it reminds to show that $\Fac \mathbf{R}(\mathcal{T}) \subset R(\mathcal{T})$.
Let $Y \in \mathbf{R}(\mathcal{T})$ and take a surjection $Y \rightarrow Z$ in $\module A^!$. 
From the previous discussion, we can assume that $Z$ has no non-zero direct summands in $\mathcal{Y}_{\epsilon_A}$. 
In this case, $Z$ lies in $\mathcal{X}_{\epsilon_A}$ since the torsion pair $(\mathcal{Y}_{\epsilon_A}, \mathcal{X}_{\epsilon_A})$ is splitting. 
Applying $\mathbf{L}$ to $Y \rightarrow Z \rightarrow 0$, we obtain an exact sequence
$\mathbf{L}(Y) \rightarrow \mathbf{L}(Z) \rightarrow 0$.
Since $\mathbf{R}|_{\mathcal{T_{\epsilon_A}}}$ and $\mathbf{L}|_{\mathcal{X}_{\epsilon_A}}$ are mutually inverse, we have 
$\mathbf{L}(Y) \in \mathbf{L}\mathbf{R}(\mathcal{T})= \mathcal{T}$.  
It implies that $\mathbf{L}(Z) \in \mathcal{T}$ since $\mathcal{T}$ is closed under factor modules. 
Since $Z\in \mathcal{X}_{\epsilon_A}$, we have $Z \cong \mathbf{R}\mathbf{L}(Z) \in \mathbf{R}(\mathcal{T}) \subseteq R(\mathcal{T})$. 
Therefore, $\Fac \mathbf{R}(\mathcal{Y}) \subset R(\mathcal{T})$ holds, and hence $R(\mathcal{T})$ is closed under factor modules. 
Secondly, we see that $R(\mathcal{T})$ is closed under extensions. 
Let 
\begin{equation} \label{seq}
0\rightarrow X \overset{\iota}{\rightarrow} Y \rightarrow Z \rightarrow 0
\end{equation}
be a short exact sequence in $\module A^!$.
We show that $X, Z \in R(\mathcal{T})$ implies $Y \in R(\mathcal{T})$.
Suppose that $X, Z \in R(\mathcal{T})$. 
Since $S^{\epsilon_A,-}_{A^!}\in \mathcal{Y}_{\epsilon_A}$ is semisimple, 
we can assume that $Y$ has no non-zero indecomposable direct summands in $\mathcal{Y}_{\epsilon_A}$. In this case,  $Y$ lies in $\mathcal{X}_{\epsilon_A}$. 
Applying $\mathbf{L}$ to (\ref{seq}), we obtain a short exact sequence
$$
0 \rightarrow \Image \mathbf{L}(\iota) \rightarrow \mathbf{L}(Y) \rightarrow \mathbf{L}(Z) \rightarrow 0 
$$
with $\mathbf{L}(X), \mathbf{L}(Z) \in \mathbf{L}(R(\mathcal{T})) \subset \mathcal{T}$. 
Since $\mathcal{T}$ is a torsion class, $\mathbf{L}(X) \in \mathcal{T}$ implies $\Image \mathbf{L}(\iota) \in \mathcal{T}$, and therefore $\mathbf{L}(Y) \in \mathcal{T}$.   
Since $Y\in \mathcal{X}_{\epsilon_A}$, we have $Y \cong \mathbf{R}\mathbf{L}(Y) \in \mathbf{R}(\mathcal{T}) \subset R(\mathcal{T})$. So $R(\mathcal{T})$ is closed under extensions.
Consequently, $R(\mathcal{T})$ is a torsion class in $\module A^!$. 
Furthermore, it is faithful by the following reason: 
Since $\mathcal{T} \in \tors_{\epsilon_A} A$, $S_{A}^{\epsilon_A,+} \in \mathcal{T}$ holds. 
By Lemma \ref{simple-injective}(3), we have 
$I_{A^!}^{\epsilon_A,+} \cong \mathbf{R}(S_{A}^{\epsilon_A,+}) \in \mathbf{R}(\mathcal{T}) \subseteq R(\mathcal{T})$.
On the other hand, by Lemma \ref{simple-injective}(4), 
we have $I_{A^!}^{\epsilon_A, -}\cong S_{A^!}^{\epsilon_A,-} \in \mathcal{Y}_{\epsilon_A} \subseteq R(\mathcal{T})$ by definition. 
Therefore, $DA^! = I_{A^!}^{\epsilon_A,+} \oplus I_{A^!}^{\epsilon_A,-} \in R(\mathcal{T})$.


Secondly, we show that $L(\mathcal{U}) := \mathbf{L}(\mathcal{U}\cap \mathcal{X}_{\epsilon_A}) \in \tors_{\epsilon_A} A$ for any faithful torsion class $\mathcal{U}$ in $\module A^!$.
Let $\mathcal{U}$ be a faithful torsion class in $\module A^!$. 
Note that $L(\mathcal{U}) \subset \mathcal{T}_{\epsilon_A}$ and 
$\mathbf{R}(L(\mathcal{U})) \subset \mathcal{U} \cap \mathcal{X}_{\epsilon_A}$. 
Moreover, $\mathcal{Y}_{\epsilon_A} \subset \mathcal{U}$ since $\mathcal{U}$ is faithful. 
Let 
\begin{equation} \label{seq2}
0\rightarrow X' \overset{\iota'}{\rightarrow} Y' \rightarrow Z' \rightarrow 0
\end{equation}
be a short exact sequence in $\module A$.

We check that $L(\mathcal{U})$ is closed under factor modules. 
Suppose that $Y' \in L(\mathcal{U}) \subset \mathcal{T}_{\epsilon_A}$. 
Since $\mathcal{T}_{\epsilon_A}$ is closed under factor modules, 
$Y' \in \mathcal{T}_{\epsilon_A}$ implies $Z' \in \mathcal{T}_{\epsilon_A}$. 
In particular, $\mathbf{R}(Z') \in \mathcal{X}_{\epsilon_A}$.
Since $\Ext^1_{A}(T[\epsilon^{-1}(-1)], \mathcal{T}_{\epsilon_A})=0$, there is a short exact sequence 
$$
0 \rightarrow \coker \mathbf{R}(\iota') \rightarrow \mathbf{R}(Z') \rightarrow 
\Ext^1_A(T[\epsilon^{-1}(-1)], X') \rightarrow 0 
$$
with $\Ext^1_A(T[\epsilon^{-1}(-1)], X') \in \mathcal{Y}_{\epsilon_A} \subset \mathcal{U}$. 
Since $\mathbf{R}(Y') \in \mathbf{R}(L(\mathcal{U})) \subset \mathcal{U}$ 
and $\mathcal{U}$ is closed under factor modules, we have $\coker \mathbf{R}(\iota') \in \mathcal{U}$. 
Since $\mathcal{U}$ is closed under extensions, $\mathbf{R}(Z') \in \mathcal{U}$, and hence $\mathbf{R}(Z') \in \mathcal{U}\cap \mathcal{X}_{\epsilon_A}$.
Finally, we have $Z' \cong \mathbf{L}\mathbf{R}(Z') \in \mathbf{L}(\mathcal{U}\cap \mathcal{X}_{\epsilon_A}) = L(\mathcal{U})$. Therefore, $L(\mathcal{U})$ is closed under factor modules.
Next, we show that $L(\mathcal{U})$ is closed under extensions. 
Suppose that $X', Z' \in L(\mathcal{U})$.
In this case, the sequence (\ref{seq2}) is an short exact sequence in $\mathcal{T}_{\epsilon_A}$.
Since $\mathbf{R}|_{\mathcal{T}_{\epsilon_A}}$ is exact, we obtain the following short exact sequence: 
$$
0 \rightarrow \mathbf{R}(X') \rightarrow \mathbf{R}(Y') \rightarrow \mathbf{R}(Z') \rightarrow 0
$$
with $\mathbf{R}(X'), \mathbf{R}(Z') \in \mathbf{R}(L(\mathcal{U})) \subset \mathcal{U} \cap \mathcal{X}_{\epsilon_A}$. 
Since both $\mathcal{U}, \mathcal{X}_{\epsilon_A}$ are closed under extensions, 
$\mathbf{R}(Y')$ is also contained in $\mathcal{U}\cap \mathcal{X}_{\epsilon_A}$.
Since $Y'\in \mathcal{T}_{\epsilon_A}$, we have $Y' \cong \mathbf{L}\mathbf{R}(Y') \in \mathbf{L}(\mathcal{U}\cap \mathcal{X}_{\epsilon_A}) = L(\mathcal{U})$. 
Thus, $L(\mathcal{U})$ is closed under extensions. 
As a consequence, $L(\mathcal{U})$ is a torsion class in $\module A$.
Moreover, we claim that $L(\mathcal{U})$ is contained in $\tors_{\epsilon_A} A$, that is, 
$L(\mathcal{U})$ satisfies $S_A^{\epsilon,+} \in L(\mathcal{U})$ and $L(\mathcal{U}) \subseteq \mathcal{T}_{\epsilon_A}$. 
By Lemma \ref{simple-injective}(4), $\mathbf{R}(S_A^{\epsilon,+}) \cong I_{A^!}^{\epsilon_A,+} \in \mathcal{X}_{\epsilon_A}$ is an injective $A^!$-module. 
Since $\mathcal{U}$ is faithful, 
$\mathbf{R}(S_A^{\epsilon,+})$ is contained in $\mathcal{U}$.
Thus, $\mathbf{R}(S_A^{\epsilon,+}) \in \mathcal{U} \cap \mathcal{X}_{\epsilon_A}$.
Since $S_A^{\epsilon,+}\in \mathcal{T}_{\epsilon_A}$, 
we have $S_A^{\epsilon,+} \cong \mathbf{L}\mathbf{R}(S_A^{\epsilon,+}) \in 
\mathbf{L}(\mathcal{U} \cap \mathcal{X}_{\epsilon_A}) =L(\mathcal{U})$.
On the other hand, $\mathcal{U}\cap \mathcal{X}_{\epsilon_A} \subseteq \mathcal{X}_{\epsilon_A}$ implies that $L(\mathcal{U}) = \mathbf{L}(\mathcal{U} \cap \mathcal{X}_{\epsilon_A}) \subseteq \mathcal{T}_{\epsilon_A}$. 
Therefore, $L(\mathcal{U}) \in \tors_{\epsilon_A} A$.

It is easy to check that $R$ and $L$ are morphisms of partially ordered sets. 
Furthermore, they are mutually inverse. In fact, we have
$$
LR(\mathcal{T})= L(\add (\mathbf{R}(\mathcal{T}), \mathcal{Y}_{\epsilon})) =
\mathbf{L}(\mathbf{R}(\mathcal{T}) \cap \mathcal{X}_{\epsilon_A}) = 
\mathbf{L}\mathbf{R}(\mathcal{T}) = \mathcal{T}
$$
for any $\mathcal{T} \in \tors_{\epsilon_A} A$, and
$$
RL(\mathcal{U}) = R(\mathbf{L}(\mathcal{U}\cap \mathcal{X}_{\epsilon_A})) =
\add (\mathbf{R}\mathbf{L}(\mathcal{U}\cap \mathcal{X}_{\epsilon_A})), \mathcal{Y}_{\epsilon_A}))= 
\add (\mathcal{U} \cap \mathcal{X}_{\epsilon_A}, \mathcal{Y}_{\epsilon_A}) = \mathcal{U}
$$
for any $\mathcal{U} \in \fators A^!$. 
We complete the proof.
\end{proof}

Next, we focus on functorially finite torsion classes. 
We begin with the following lemma.

\begin{lemma}
\label{radical}
Let $M \in \mathcal{T}_{\epsilon_A}$ and $Q \in \mathcal{F}_{\epsilon_A}$. 
Then the following conditions are equivalent. 
\begin{enumerate}
\item[(a)] $\Hom_{A}(Q, M)=0$ 
\item[(b)] $\Ext^1_{A^{!}}(Q_{A^!}, \mathbf{R}(M))=0$.
\end{enumerate}
\end{lemma}

\begin{proof}
(b) $\Rightarrow$ (a): Without loss of generality, we can assume that $Q$ is simple. 
Let $f \colon Q \rightarrow M$ be a non-zero morphism of $A$-modules.  
Since $Q$ is simple, $f$ is injective. 
Since $\mathcal{T}_{\epsilon_A}$ is closed under factor modules, 
$M \in \mathcal{T}_{\epsilon_A}$ implies $N:=\coker f \in \mathcal{T}_{\epsilon_A}$. 
Applying $\mathbf{R}$ to a short exact sequence $0 \rightarrow Q \rightarrow M \rightarrow N \rightarrow 0$, we obtain the following short exact sequence in $\module A^{!}$:
$$
0 \rightarrow \mathbf{R}(M) \rightarrow \mathbf{R}(N) \rightarrow \Ext^1_{A}(T[\epsilon^{-1}(-1)], Q) \cong Q_{A^!} \rightarrow 0.  
$$
It does not splits since $\mathbf{R}(N) \in \mathcal{X}_{\epsilon_A}$ has no non-zero direct summands in $\mathcal{Y}_{\epsilon_A}$. 
Thus, we have the assertion. 

(a) $\Rightarrow$ (b): Dually, we get the converse.
\end{proof}

\begin{lemma} \label{stau-T-F}
Let $(M,Q)$ be a $\tau$-tilting pair for $A$. Then it is contained in $\stau_{\epsilon_A} A$ if and only if $M\in \mathcal{T}_{\epsilon_A}$ and $Q\in \mathcal{F}_{\epsilon_A}$.
\end{lemma}

\begin{proof}
Let $(M,Q)$ be a $\tau$-tilting pair for $A$. 
Let $0\rightarrow P^{-1} \rightarrow P^0 \rightarrow M \rightarrow 0$ be a minimal projective presentation of $M$ in $\module A$.
By Proposition \ref{restriction-epsilon}, 
$(M,Q) \in \stau_{\epsilon_A} A$ if and only if 
$\add P^{0} =\add P_A^{\epsilon_A,+}$ and 
$\add (Q\oplus P^{-1}) = \add P_A^{\epsilon_A,-} =  \mathcal{F}_{\epsilon_A}$.

If $(M,Q) \in \stau_{\epsilon_A} A$, then we have 
$M \in \Fac P^0 \subseteq \mathcal{T}_{\epsilon_A}$ and 
$Q \in \mathcal{F}_{\epsilon_A}$ from the above argument. 
Conversely, we assume that $M \in \mathcal{T}_{\epsilon_A}$ and $Q \in \mathcal{F}_{\epsilon_A}$. 
Clearly, we have $P^0 \in \add P_A^{\epsilon,+}$.
In addition, $P^{-1} \in \add (\rad P^0) \subseteq  \mathcal{F}_{\epsilon_A}$ by Lemma \ref{simple-injective}(2). So, $P^{-1}\oplus Q \in \mathcal{F}_{\epsilon_A}$. 
Consequently, $(M,Q) \in \stau_{\epsilon_A} A$.
\end{proof}

Our result is the following.

\begin{proposition} \label{stau-tilt}
The following diagram is commutative, and all arrows are isomorphisms of partially ordered sets:
$$
\xymatrix@C50pt{
\ar@{}[rd]|{\circlearrowright}
\stau_{\epsilon_A} A \ar[r]^{\varphi} \ar[d]^{\Fac} & \tilt A^! \ar[d]^{\Fac} \\
\ftors_{\epsilon_A} A  \ar[r]^{R} & \ffators A^!  
}
$$
where $R$ is the map in Theorem \ref{tors-fators}, and $\varphi$ is defined by $\varphi(M, Q) := \mathbf{R}(M) \oplus Q_{A^!}$ for any 
$\tau$-tilting pair $(M, Q) \in \stau_{\epsilon_A} A$. 
\end{proposition}

\begin{proof}
We remark that an $A$-module $M$ is $\tau$-rigid if and only if it is rigid 
by Proposition \ref{hereditary tilting}(1) since $A$ is hereditary. 

Firstly, we prove that $\varphi(M,Q):=\mathbf{R}(M)\oplus Q_{A^!}$ is a tilting $A^!$-module 
for any $\tau$-tilting pair $(M,Q) \in \stau_{\epsilon_A} A$. 
Let $(M,Q) \in \stau_{\epsilon_A} A$ be a $\tau$-tilting pair. 
We need to check that $\mathbf{R}(M)\oplus Q_{A^!}$ is rigid and $|\mathbf{R}(M)|+ |Q_{A^!}|=n$. 
By Lemma \ref{stau-T-F}, we have $M \in \mathcal{T}_{\epsilon_A}$ and $Q\in \mathcal{F}_{\epsilon_A}$. 
From the above remark, $\Ext^1_A(M,M)=0$ holds since $M$ is ($\tau$-)rigid.   
Since $\mathbf{R}|_{\mathcal{T}_{\epsilon_A}}$ is exact, 
we have $\Ext^1_{A^!}(\mathbf{R}(M), \mathbf{R}(M))=0$. 
On the other hand, by Lemma \ref{radical}, $\Hom_{A}(Q,M)=0$ implies $\Ext^1_{A^!}(Q_{A^!}, \mathbf{R}(M))=0$.
Moreover, we have $\Ext^1_{A^!}(\mathbf{R}(M)\oplus Q_{A^!}, Q_{A^!})=0$ since $Q_{A^!} \in \mathcal{Y}_{\epsilon_A}$ is injective. 
Thus, $\mathbf{R}(M) \oplus Q_{A^!}$ is rigid. 
On the other hand, we have $|\mathbf{R}(M)|+ |Q_{A^!}|= |M| + |Q| = n$ 
because a pair $(M,Q)$ is 
$\tau$-tilting. Therefore, $\varphi(M,Q):=\mathbf{R}(M)\oplus Q_{A^!}$ is tilting. 

Next, we give an inverse map of $\varphi$. 
For every $T\in \tilt A$, let $\psi(T):=(\mathbf{L}(T'), T''_A)$ where 
$T=T'\oplus T''$ is a decomposition of $T$ into direct summands $T',T''$ such that $T'\in \mathcal{X}_{\epsilon_A}$ and $T'' \in \mathcal{Y}_{\epsilon_A}$. 
We show that $\psi(T)$ provides a $\tau$-tilting pair for $A$. 
Since $\mathbf{L}|_{\mathcal{X}_{\epsilon}}$ is exact, we have 
$\Ext^1_{A}(\mathbf{L}(T'), \mathbf{L}(T')) \cong \Ext^1_{A^!}(T', T') =0$.
So, $\mathbf{L}(T')$ is ($\tau$-)rigid. 
On the other hand, by Lemma \ref{radical}, 
$\Ext^1_{A^!}(T'', \mathbf{R}(\mathbf{L}(T'))) = \Ext^1_{A^!}(T'', T') =0$ implies $\Hom_A(T''_A, \mathbf{L}(T'))=0$. Thus, $(\mathbf{L}(T'), T'')$ is a $\tau$-rigid pair. 
Finally, we have $|\mathbf{L}(T')| + |T''_A| = |T'| + |T''|= |T|=|A^!| =n$.
Therefore, $\psi(T)=(\mathbf{L}(T'), T''_A)$ is a $\tau$-tilting pair.
By Lemma \ref{stau-T-F}, it is contained in $\stau_{\epsilon_A} A$ since 
$\mathbf{L}(T')\in \mathcal{T}_{\epsilon_A}$ and $T''_A \in \mathcal{F}_{\epsilon_A}$.
Consequently, we obtain a map $\psi \colon \tilt A^! \rightarrow \stau_{\epsilon_A} A$.

The maps $\varphi$ and $\psi$ are mutually inverse. In fact, we have 
$
\psi(\varphi(M, Q))=\psi(\mathbf{R}(M) \oplus Q_{A^!}) = (\mathbf{L}\mathbf{R}(M), Q_A)
=(M, Q)
$ 
and 
$\varphi(\psi(T)) = \varphi(\mathbf{L}(T'), T''_A) = 
\mathbf{R}\mathbf{L}(T') \oplus T''_{A^!} = T$
for any $(M,P) \in \stau_{\epsilon_A} A$ and $T \in \tilt A^!$. 

Next, we show the commutativity of the diagram in our statement.
Let $(M,Q)\in \stau_{\epsilon_A} A$ be a $\tau$-tilting pair.
We need to show that 
$\mathbf{R}(\Fac M):=\add (\mathbf{R}(\Fac M), \mathcal{Y}_{\epsilon})$ is equal to 
$\Fac \varphi (M,Q)$. 
Clearly, $\varphi(M,Q)=\mathbf{R}(M) \oplus Q_{A^!} \in \add (\mathbf{R}(\Fac M), \mathcal{Y}_{\epsilon})$, so we have 
$ \Fac \varphi (M,Q)  \subseteq \add (\mathbf{R}(\Fac M), \mathcal{Y}_{\epsilon})$.
To prove the converse, we are enough to check that 
$\mathbf{R}(\Fac M)$ and $\mathcal{Y}_{\epsilon}$ are contained in $\Fac \varphi(M,Q)$. 
Since $\varphi(M,Q)$ is tilting, we have $\mathcal{Y}_{\epsilon} \in \Fac \varphi(M,Q)$. 
On the other hand, let $X \in \Fac M$ and consider a short exact sequence 
$0 \rightarrow N \rightarrow  M^{m} \rightarrow X \rightarrow 0$ for some integer $m$.  
Applying $\mathbf{R}$ to this sequence, we have 
$$
0 \rightarrow \mathbf{R}(N)  \overset{f}{\rightarrow} \mathbf{R}(M)^{m} \rightarrow \mathbf{R}(X) \rightarrow \Ext^1_{A^!}(T[\epsilon^{-1}(-1)], N) \rightarrow 0.
$$ 
with $\coker f, \Ext^1_{A^!}(T[\epsilon^{-1}(-1)], N) \in \Fac \varphi(M,Q)$. 
Since $\Fac \varphi(M,Q)$ is a torsion class, it implies that 
$\mathbf{R}(X) \in \Fac \varphi (M,Q)$. Thus, we have $\mathbf{R}(\Fac M) \subset \Fac \varphi(M,Q)$. Consequently, we have $R(\Fac M) = \Fac \varphi (M,Q)$.

Finally, since the maps $\Fac$ and $R$ in the diagrams are isomorphism of partially ordered sets, $\varphi$ is also an isomorphism.
We complete the proof.
\end{proof}

\begin{example} \rm
Let $Q$ be a finite connected quiver. If $Q$ is acyclic and bipartite, then the path algebra $A:=kQ$ is a hereditary algebra with radical square zero, whose (valued) quiver is $Q$.  
By (\ref{A! as upper matrix algebra}), we find that the algebra $A^!$ is isomorphic to $kQ^{\rm op}$. 
 
Let $Q \colon 1 \longleftarrow 2 \longrightarrow 3$ be a quiver. Then 
the path algebra $A=kQ$ of $Q$ is a hereditary algebra with radical square zero. 
And the associated map $\epsilon_A \colon [3] \rightarrow \{\pm1\}$ is given by 
$\epsilon_A(1)=-1, \epsilon_A(2)=1, \epsilon_A(3)=-1$. 
From the above remark, there is an isomorphism of algebras $A^!\cong kQ^{\rm op}$. 
The set $\tors_{\epsilon_A} A$ consists of five torsion classes of $\module A$, which are all functorially finite.   
In the following table, we describe the bijection in Theorem \ref{hereditary case} between $\tors_{\epsilon_A} A$ and $\fators A^!$. 
Here, each black point in the Auslander-Reiten quiver corresponds to an indecomposable module contained in a given torsion class, 
and squared one in the figure corresponds to a simple module in $\mathcal{F}_{\epsilon_A}$ or in $\mathcal{Y}_{\epsilon_A}$:

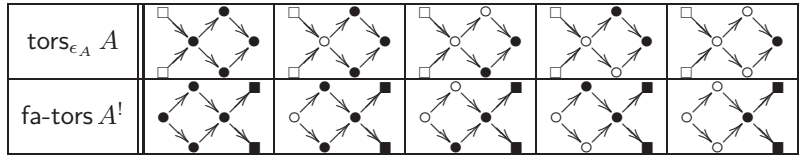
\begin{table}[h] 
{\renewcommand{\arraystretch}{2}
  \begin{tabular}{|c||c|c|c|c|c|}
\hline $\tors_{\epsilon_A}A$ & 
$\begin{xy} 
(0,0)*{\bullet}="3",  (-4, 4)*{\scalebox{0.6}{$\square$}}="1", 
(-4, -4)*{\scalebox{0.6}{$\square$}}="2", 
(4, 4)*{\bullet}="4",  (4, -4)*{\bullet}="5",  (8, 0)*{\bullet}="6", 
{"1" \ar "3"},
{"2" \ar "3"},
{"3" \ar "4"},
{"3" \ar "5"},
{"4" \ar "6"},
{"5" \ar "6"},
\end{xy}$ 
& 
$\begin{xy} 
(0,0)*{\circ}="3",  (-4, 4)*{\scalebox{0.6}{$\square$}}="1", 
(-4, -4)*{\scalebox{0.6}{$\square$}}="2", 
(4, 4)*{\bullet}="4",  (4, -4)*{\bullet}="5",  (8, 0)*{\bullet}="6", 
{"1" \ar "3"},
{"2" \ar "3"},
{"3" \ar "4"},
{"3" \ar "5"},
{"4" \ar "6"},
{"5" \ar "6"},
\end{xy}$ 
&
$\begin{xy} 
(0,0)*{\circ}="3",  (-4, 4)*{\scalebox{0.6}{$\square$}}="1", 
(-4, -4)*{\scalebox{0.6}{$\square$}}="2", 
(4, 4)*{\circ}="4",  (4, -4)*{\bullet}="5",  (8, 0)*{\bullet}="6", 
{"1" \ar "3"},
{"2" \ar "3"},
{"3" \ar "4"},
{"3" \ar "5"},
{"4" \ar "6"},
{"5" \ar "6"},
\end{xy}$ 
&
$\begin{xy} 
(0,0)*{\circ}="3",  (-4, 4)*{\scalebox{0.6}{$\square$}}="1", 
(-4, -4)*{\scalebox{0.6}{$\square$}}="2", 
(4, 4)*{\bullet}="4",  (4, -4)*{\circ}="5",  (8, 0)*{\bullet}="6", 
{"1" \ar "3"},
{"2" \ar "3"},
{"3" \ar "4"},
{"3" \ar "5"},
{"4" \ar "6"},
{"5" \ar "6"},
\end{xy}$ 
&
$ \begin{xy} 
(0,0)*{\circ}="3",  (-4, 4)*{\scalebox{0.6}{$\square$}}="1", 
(-4, -4)*{\scalebox{0.6}{$\square$}}="2", 
(4, 4)*{\circ}="4",  (4, -4)*{\circ}="5",  (8, 0)*{\bullet}="6", 
{"1" \ar "3"},
{"2" \ar "3"},
{"3" \ar "4"},
{"3" \ar "5"},
{"4" \ar "6"},
{"5" \ar "6"},
\end{xy} $ \\ 
\hline 
$\fators A^!$
&
$\begin{xy} 
(0,0)*{\bullet}="3",  (12, 4)*{\scalebox{0.6}{$\blacksquare$}}="1", 
(12, -4)*{\scalebox{0.6}{$\blacksquare$}}="2", 
(4, 4)*{\bullet}="4",  (4, -4)*{\bullet}="5",  (8, 0)*{\bullet}="6", 
{"6" \ar "1"},
{"6" \ar "2"},
{"3" \ar "4"},
{"3" \ar "5"},
{"4" \ar "6"},
{"5" \ar "6"},
\end{xy}$ 
&
$\begin{xy} 
(0,0)*{\circ}="3",  
(12, 4)*{\scalebox{0.6}{$\blacksquare$}}="1", 
(12, -4)*{\scalebox{0.6}{$\blacksquare$}}="2", 
(4, 4)*{\bullet}="4",  (4, -4)*{\bullet}="5",  (8, 0)*{\bullet}="6", 
{"6" \ar "1"},
{"6" \ar "2"},
{"3" \ar "4"},
{"3" \ar "5"},
{"4" \ar "6"},
{"5" \ar "6"},
\end{xy}$ 
&
$\begin{xy} 
(0,0)*{\circ}="3",  
(12, 4)*{\scalebox{0.6}{$\blacksquare$}}="1", 
(12, -4)*{\scalebox{0.6}{$\blacksquare$}}="2", 
(4, 4)*{\circ}="4",  (4, -4)*{\bullet}="5",  (8, 0)*{\bullet}="6", 
{"6" \ar "1"},
{"6" \ar "2"},
{"3" \ar "4"},
{"3" \ar "5"},
{"4" \ar "6"},
{"5" \ar "6"},
\end{xy}$ 
&
$\begin{xy} 
(0,0)*{\circ}="3", 
(12, 4)*{\scalebox{0.6}{$\blacksquare$}}="1", 
(12, -4)*{\scalebox{0.6}{$\blacksquare$}}="2", 
(4, 4)*{\bullet}="4",  (4, -4)*{\circ}="5",  (8, 0)*{\bullet}="6", 
{"6" \ar "1"},
{"6" \ar "2"},
{"3" \ar "4"},
{"3" \ar "5"},
{"4" \ar "6"},
{"5" \ar "6"},
\end{xy}$ 
&
$\begin{xy} 
(0,0)*{\circ}="3", 
(12, 4)*{\scalebox{0.6}{$\blacksquare$}}="1", 
(12, -4)*{\scalebox{0.6}{$\blacksquare$}}="2", 
(4, 4)*{\circ}="4",  (4, -4)*{\circ}="5",  (8, 0)*{\bullet}="6", 
{"6" \ar "1"},
{"6" \ar "2"},
{"3" \ar "4"},
{"3" \ar "5"},
{"4" \ar "6"},
{"5" \ar "6"},
\end{xy}$ \\ \hline 
  \end{tabular}
}
\end{table}
\end{example}

\subsection{$g$-vectors and dimension vectors}
In this section, we show that the bijection $\varphi$ in Proposition \ref{stau-tilt} induces a $\mathbb{Z}$-linear transformation between the corresponding $g$-vectors and dimension vectors.
Let $A$ be a connected hereditary algebra with radical square zero and let $A^!$ be an algebra defined in Section \ref{torsion class for hereditary}. 
As in the previous section, we fix a natural bijection 
between $\mathcal{S}(A)=\{e_1,\ldots, e_n\}$ and $\mathcal{S}(A^!)=\{e'_1,\ldots,e'_n\}$ mapping $e_i\mapsto e'_i$ for all $i\in [n]$. 
We promise that $g$-vectors and dimension vectors of $A$-modules and $A^!$-modules 
are considered with respect to these indexes. 
Let $\mathrm{B}_{\epsilon_A}:={\rm diag}(\epsilon_A(1), \ldots, \epsilon_A(n))$ be a diagonal matrix. We prove the following.

\begin{proposition}\label{g-c-vector} 
For any $\tau$-tilting pair $(M, Q) \in \stau_{\epsilon_A} A$, we have $g_{A}^{(M, Q)}= 
\mathrm{B}_{\epsilon_A} \cdot c_{A^!}^{\varphi(M, Q)}$.
\end{proposition}

We need the following equations.
Let $\mathrm{S}_{\epsilon_A}:=\mathrm{S}_{\epsilon_A^{-1}(-1)}$ be a square matrix defined in Section \ref{valued quiver}, and let $\mathrm{I}_n$ be the identity matrix of size $n$. 

\begin{lemma} \label{matrices}
\begin{enumerate}
\item  $\mathrm{S}_{\epsilon_A}=\mathrm{C}_{A} \mathrm{B}_{\epsilon_A}$. 
\item  $2\, \mathrm{I}_n-\mathrm{C}_{A}=\mathrm{B}_{\epsilon_A} \mathrm{C}_{A} \mathrm{B}_{\epsilon_A}$.
\item $\mathrm{S}_{\epsilon_A}^2 = \mathrm{B}_{\epsilon_A}^2 = \mathrm{I}_n$
\end{enumerate}
\end{lemma}

\begin{proof}
Let $\Gamma_A$ be a valued quiver of $A$. 
Recall that each valuation $d_{ij}'$ of an arrow $i \rightarrow j$ is equal to 
$j$-th entry of the dimension vector of $\rad P(i)/ \rad^2 P(i)$. 
Since $A$ is a hereditary algebra with radical square zero,
this is equal to the number $(\mathrm{C}_A)_{ji}=(c_A^{\rad P(i)})_j$ for any $i \neq j$. 

Reordering the numbering of vertices if you need, we can assume that $\epsilon^{-1}(1)=\{1, \ldots, m\}$ is the set of sources and $\epsilon^{-1}(-1)=\{m+1, \ldots, n\}$ is the set of sinks in $\Gamma_A$.
In this situation, our matrices are given by 
$$
\mathrm{C}_A = \begin{pmatrix}
\mathrm{I}_m & 0 \\ 
\mathrm{D} & \mathrm{I}_{n-m}
\end{pmatrix}, \ 
\mathrm{S}_{\epsilon_A} = \begin{pmatrix}
\mathrm{I}_m & 0 \\ 
\mathrm{D} & -\mathrm{I}_{n-m}
\end{pmatrix} \  and \ 
\mathrm{B}_{\epsilon_A} = \begin{pmatrix}
\mathrm{I}_m & 0 \\ 
0 & -\mathrm{I}_{n-m}
\end{pmatrix} \  
$$
where $\mathrm{D}=(d_{ji}')_{ij}$ is a matrix of size $(n-m) \times m$.
From an easy calculation, we get equations (1), (2) and (3). 
\end{proof}

The following calculation is an analog of \cite{Ben} in the study of the projective resolution over symmetric algebras with radical cube zero.

\begin{lemma} 
\label{Benson}
For an $A$-module $M \in \mathcal{T}_{\epsilon_A}$, we have 
$g_{A}^M= \mathrm{B}_{\epsilon_A} \mathrm{C}_A \mathrm{B}_{\epsilon_A} \cdot c_{A}^M$.
\end{lemma}

\begin{proof}
We start with a few basic observations of dimension vectors. 
Let $M\in \mathcal{T}_{\epsilon_A}$. 
From the short exact sequence $0\rightarrow \rad M \rightarrow M \rightarrow \tops M \rightarrow 0$, we obtain an equality $c_A^{M}=c_A^{\tops M}+c_A^{\rad M}$. 
On the other hand, let $0 \rightarrow P^{-1} \rightarrow P^0 \overset{f}{\rightarrow} M \rightarrow 0$ be a minimal projective presentation of $M$. 
With respect to the index $e_1,\ldots, e_n$, we have 
$g^{P^0}_A=c_A^{\tops P^0}= c_A^{\tops M}$. 
By definition of Cartan matrix of $A$, we have $c_A^{P^0} = \mathrm{C}_A \cdot c_{A}^{\tops P^0}= \mathrm{C}_A \cdot c_{A}^{\tops M}$. 
Therefore, $c_A^{\rad P^0} = (\mathrm{C}_A - \mathrm{I}_n) c_A^{\tops M}$ holds.

In our case, $\rad P^0$ is a semisimple $A$-module contained in $\mathcal{F}_{\epsilon_A}$ by Lemma \ref{simple-injective}(2). 
It implies that the short exact sequence 
$0\rightarrow \ker f \rightarrow \rad P^0 \rightarrow \rad M \rightarrow 0$ splits. Therefore, $g^{P^{-1}}_A = c^{\ker f}_A = c^{\rad P^0}_A - c^{\rad M}_A$ holds.

Consequently, we have 
\begin{eqnarray*}
g^M = g^{P^0} -g^{P^{-1}} &=& 
c_A^{\tops M} - c^{\rad P^0}_A + c^{\rad M}_A = 
c_A^{\tops M}  - (\mathrm{C}_A - \mathrm{I}_n) c_A^{\tops M} + c^{\rad M}_A \\
&=& 
(2\mathrm{I}_n - \mathrm{C}_A) (c_A^{\tops M} + c^{\rad M}_A) \\
&\overset{\rm Lem. \ref{matrices}(2)}{=} & \mathrm{B}_{\epsilon_A} \mathrm{C}_A \mathrm{B}_{\epsilon_A}\cdot c_A^M, 
\end{eqnarray*}
where we use $(2\mathrm{I}_n - \mathrm{C}_A) \cdot c_A^{\rad M} 
=c_A^{\rad M}$ since $\rad M \in \mathcal{F}_{\epsilon_A}$.
Thus, the assertion holds.
\end{proof}

Now, we are ready to prove Proposition \ref{g-c-vector}.

\begin{proof}[Proof of Proposition \ref{g-c-vector}]
If $A$ is simple, it is clear. 
So we assume that $A$ is non-simple. 
Let $(M, Q) \in \stau_{\epsilon_A} A$ be a $\tau$-tilting pair.   
By Lemma \ref{stau-T-F}, we have $M \in \mathcal{T}_{\epsilon_A}$ and $Q \in \mathcal{F}_{\epsilon_A}$. 
Since $Q$ is semisimple projective, we have $g_{A}^{Q} = c_A^{Q}=c_{A^!}^{Q_{A^!}}$. 
On the other hand, $c_{A}^{M} = \mathrm{S}_{\epsilon_A} \cdot c_{A^!}^{\mathbf{R}(M)}$ holds by Proposition \ref{reflection}(d).  
So, we have
\begin{eqnarray*}
 g_{A}^{(M, Q)} &=& g_{A}^M - g_{A}^Q
\overset{\rm Lem.\ref{Benson}}{=} {\rm B}_{\epsilon_A} \mathrm{C}_A \mathrm{B}_{\epsilon_A} \cdot c_{A}^M - c_{A^!}^{Q_{A^!}}  \overset{\rm Lem.\ref{matrices}(1)}{=} {\rm B}_{\epsilon_A}\mathrm{S}_{\epsilon_A}^2 \cdot c_{A^!}^{\mathbf{R}(M)} - c_{A^!}^{Q_{A^!}} 
\\ &\overset{\rm Lem. \ref{matrices}(3)}{=}& {\rm B}_{\epsilon_A} (c_{A^!}^{\mathbf{R}(M)} + c_{A^!}^{Q_{A^!}})  
= {\rm B}_{\epsilon_A}\cdot c_{A^!}^{\varphi(M, Q)}. 
\end{eqnarray*} 
We obtain the desired equality.
\end{proof}

\subsection{Another approach to $\varphi$} \label{two-silt-version}
The aim of this subsection is to understand the bijection $\varphi$ in Proposition \ref{stau-tilt}, in the level of the homotopy category. 
Recall that $\Ktwo(\proj A)$ is a full subcategory of $\Kb(\proj A)$ consisting of two-term complexes $T$ such that 
$g_A^T \in \{x \in \mathbb{Z}^n \mid 
\text{$x_i \in \epsilon(i) \cdot \mathbb{Z}_{\geq0}$ for all $i \in [n]$\}}$. 
Our result is the following.

\begin{theorem} \label{Koszul duality}
Let $A$ be a connected hereditary algebra with radical square zero. 
Then, there is an equivalence of additive categories 
$$
\xi \colon \module A^! \overset{\sim}{\longrightarrow} 
\mathsf{K}^{\rm 2}_{\epsilon_A}(\proj A) 
$$
and it induces the following commutative diagram:
$$
\begin{xy}
(0,0)*+{\twosilt_{\epsilon_A} A}="1", (50,0)*+{\stau_{\epsilon_A} A}="2", 
(25, -13)*+{\tilt A^!}="3", 
(25, -5)*{\circlearrowright},
{"1" \ar^{H^0} "2"},
{"2" \ar^{\varphi} "3"},  
{"1" \ar_{\xi^{-1}} "3"}, 
\end{xy}
$$
\end{theorem}

Let $A$ be a connected hereditary algebra with radical square zero. 
We start with easy lemma. 

\begin{lemma} \label{Kb=Cb}
For any two-term complexes $T,U\in \mathsf{K}^{\rm 2}_{\epsilon_A}(\proj A)$, we have 
$$
\Hom_{\Kb(\proj A)}(T,U) \cong \Hom_{\mathsf{C}^{\rm b}(\proj A)}(T,U).
$$
\end{lemma}

\begin{proof}
We show that any morphism between complexes which is null-homotopic is zero.  
Let $T=(T^{-1} \overset{d_T}{\rightarrow} T^{0})$ and $U=(U^{-1} \overset{d_U}{\rightarrow} U^{0})$ be complexes in $\mathsf{K}^{\rm 2}_{\epsilon_A}( \proj A)$.  
Up to isomorphism, we can assume each differential $d_T, d_U$ is in the radical of $\proj A$. 
Assume that $(f^{-1}, f^0) \in \Hom_{\Kb(\proj A)} (T,U)$ is null-homotopic. 
By definition, there exists a morphism $h \colon T^0 \rightarrow U^{-1}$ in $\proj A$ such that $f^{-1}=h \circ d_T$ and $f^{0}= d_U \circ h$. 
Since both of $T$ and $U$ are in $\mathsf{K}^{\rm 2}_{\epsilon_A}(\proj A)$, we have  $\add T^0 \cap \add U^{-1}=0$, and hence $h$ is in the radical of the category $\proj A$.  
Since $A$ is an algebra with radical square zero, 
$f^{-1}=h \circ d_T = 0$ and $f^{0}= d_U \circ h = 0$ hold. Thus, the assertion holds.
\end{proof}

Now, for each $\sigma \in \{\pm\}$, let $F^{\epsilon_A, \sigma} := e^{\epsilon_A, \sigma} (A/J_A) e^{\epsilon_A, \sigma}$ be a semisimple algebra.  
Recall that $A^!$ is defined as a upper triangular matrix algebra 
$\begin{pmatrix} F^{\epsilon_A, -} & DJ_A \\ 0 & F^{\epsilon_A, +} \end{pmatrix}$.  
According to \cite[Section III.2]{ARS}, each $A^{!}$-module $X$ can be 
written as a triple $(X^-, X^+, f_X)$ such that $X^+ \in \module F^{\epsilon_A,+}$, 
$X^- \in \module F^{\epsilon_A,-}$ 
and $f_X \colon X^- \otimes_{F^{\epsilon_A,-}} DJ_A \rightarrow X^+$ is a morphism in 
$\module F^{\epsilon_A,+}$.
Next lemma is useful for the vanishing condition of the $\Ext$-group.

\begin{lemma} \label{Gab-Roi} 
In the above, there is a short exact sequence: 
$$
0 \rightarrow \Hom_{A^!}(X,Y) \rightarrow 
\bigoplus_{\sigma \in \{\pm\}} \Hom_{F^{\epsilon_A, \sigma}}(X^{\sigma}, Y^{\sigma}) \overset{\delta}{\rightarrow} 
\Hom_{F^{\epsilon_A, +}}(X^-\otimes_{F^{\epsilon_A, +}}DJ_A, Y^+) \rightarrow 
\Ext^1_{A^!}(X,Y) \rightarrow0,
$$
where $\delta$ maps  $h=(h^-, h^+)$  
onto $h^+f_X-f_Y(h^-\otimes1_{DJ_A})$.
\end{lemma}

\begin{proof}
We just apply the result \cite[Section 7]{GR} for triples $(X^-, X^+, f_X)$ and $(Y^-, Y^+, f_Y)$.
\end{proof}

On the other hand, there is an equivalence of additive categories for each $\sigma\in \{\pm1\}$: 
\begin{equation} \label{S-P}
-\otimes_{F^{\epsilon_A, \sigma}} P_A^{\epsilon_A, \sigma} \colon \module F^{\epsilon_A, \sigma} \overset{\sim}{\longrightarrow} \add P_A^{\epsilon_A, \sigma}.
\end{equation}

A key observation is the following.

\begin{lemma} \label{simple proj}
For any $M \in \mod F^{\epsilon_A,+}$ and $N \in \module F^{\epsilon_A, -}$, we have an isomorphism of bifunctors
\begin{equation}\label{homo}
\widehat{(-)}: \Hom_{F^{\epsilon_A,+}}(N \otimes_{F^{\epsilon_A,-}} DJ_A, M) 
\overset{\sim}{\longrightarrow} \Hom_{A}(N \otimes_{F^{\epsilon_A,-}}P_{A}^{\epsilon_A,-}, M\otimes_{F^{\epsilon_A,+}} P_{A^!}^{\epsilon_A,+}). 
\end{equation}
\end{lemma}

\begin{proof}
Let $M \in \mod F^{\epsilon_A,+}$ and $N \in \module F^{\epsilon_A,-}$. 
Note that $J_A= e^{\epsilon_A, +} A e^{\epsilon_A, -}$ holds. 
By the hom-tensor adjoint, we have 
\begin{eqnarray*}
\Hom_{F^{\epsilon_A,+}}(N\otimes_{F^{\epsilon_A,-}} DJ_A, M) &\cong& \Hom_{F^{\epsilon_A,-}}(N, M \otimes_{F^{\epsilon_A,+}} J_A) 
\\ &\cong& \Hom_{F^{\epsilon_A,-}}(N, \Hom_{A}(e^{\epsilon_A,-} A, M\otimes_{F^{\epsilon_A,+}} e^{\epsilon_A,+}A)) \\ 
&\cong& \Hom_{A}(N \otimes_{F^{\epsilon_A,-}} P_{A}^{\epsilon_A,-} , M \otimes_{F^{\epsilon_A,+}}P_{A^!}^{\epsilon_A,+}).
\end{eqnarray*}
Thus, we obtain the desired isomorphism.
It is easy to check that it is a natural transformation.
\end{proof}

Now, we are ready to prove Theorem \ref{Koszul duality}.

\begin{proof}[Proof of Theorem \ref{Koszul duality}]
We construct an equivalence $\xi$. 
For any $X\in \module A^{!}$, we set a two-term complex
$$
\xi(X) :=((X^-\otimes_{F^{\epsilon_A,-}}P_{A}^{\epsilon_A,-}) \overset{\widehat{f_X}}{\longrightarrow} (X^+ \otimes_{F^{\epsilon_A,+}}P_{A}^{\epsilon_A,+})) \in \mathsf{K}^{\rm 2}_{\epsilon_A}(\proj A), 
$$
where $(X^-,X^+, f_X)$ is the corresponding triple to $X$, and
$\widehat{f_X}$ is a morphism given by Lemma \ref{simple proj}. 

Let $X, Y \in \module A^!$. 
By Lemma \ref{Kb=Cb}, we have 
$\Hom_{\Kb(\proj A)}(\xi(X), \xi(Y))=\Hom_{\mathsf{C}^{\rm b}(\proj A)}(\xi(X), \xi(Y))$.
By Lemma \ref{exa} and \ref{Gab-Roi}, 
we obtain the following diagram: 
$$\hspace{-14mm} 
\scalebox{0.85}{
\xymatrix@C14pt
{
0  \ar[r] 
& \Hom_{A^!}(X,Y) \ar[r] \ar@{.>}[d]^{\xi}_{\rotatebox{90}{$\sim$}}
& {\displaystyle\bigoplus_{\sigma\in \{\pm\}}} \Hom_{F^{\epsilon_A, \sigma}}(X^\sigma, Y^\sigma) \ar[r] \ar[d]_{\rotatebox{90}{$\sim$}}^{(\ref{S-P})} \ar@{}[rd]|{\circlearrowright}
& \Hom_{F^{\epsilon_A,+}}(X^- \otimes_{F^{\epsilon_A,-}}DJ_A, Y^+) \ar[r] \ar[d]_{\rotatebox{90}{$\sim$}}^{\widehat{(-)}} 
& \Ext^1_{A^!}(X,Y) \ar[r] \ar@{.>}[d]_{\rotatebox{90}{$\sim$}}
& 0\\ 
0  \ar[r] 
& \Hom_{\Kb(\proj A)}(\xi(X),\xi(Y)) \ar[r]  
& {\displaystyle \bigoplus_{i\in \{-1,0\}}} \Hom_{A}(\xi(X)^i, \xi(Y)^i) \ar[r] 
& \Hom_{A}(\xi(X)^{-1}, \xi(Y)^0) \ar[r]  
& \Hom_{\Kb(\proj A)}(\xi(X), \xi(Y)[1]) \ar[r]  
& 0
}}
$$  
Since the square in center commutes,
the map $\xi$ in the above diagram is induced, 
and it gives rise to a fully faithful functor $\xi \colon \module A^! \rightarrow \mathsf{K}^{\rm 2}_{\epsilon_A}(\proj A)$. 
Moreover, it is dense by Lemma \ref{simple proj}. Therefore, $\xi$ is an equivalence.  

On the other hand, from the above diagram, we have 
$$
\Ext^1_{A^!}(X,Y) \cong \Hom_{\Kb(\proj A)}(\xi(X), \xi(Y)[1]) 
$$
for any $X, Y \in \module A^!$. 
Since $\xi$ gives a bijection of isomorphism classes of indecomposable objects, 
it induces a bijection $\xi \colon \tilt A^! \rightarrow \twosilt_{\epsilon_A} A$  
by Proposition \ref{properties of two-silt}(1).  
Furthermore, $g_{A}^{\xi(X)} = {\rm B}_{\epsilon_A} \cdot c_{A^!}^X$ holds for any $X \in \module A^!$ by our construction. 
The commutativity of the diagram in the assertion follows from Proposition \ref{properties of g-vectors}(1) and Theorem \ref{g-c-vector}. 
\end{proof}

\begin{remark} \rm \label{extend to non-connected}
We can naturally extend Theorem \ref{hereditary case} and Proposition \ref{stau-tilt} to non-connected hereditary algebras $A$ with radical square zero in the following way: 
Suppose that $\mathcal{S}(A):=\{e_1,\ldots, e_n\}$ be a complete set of primitive orthogonal idempotents of $A$ and $\Gamma_A$ a valued quiver of $A$ with vertex set $1,\ldots,n$. 
Each vertex $i\in [n]$ of $\Gamma_A$ is either a source, a sink or isolated. 
Take a map $\epsilon_A \colon [n] \rightarrow \{\pm1\}$ such that  
$\epsilon(i)=1$ if $i$ is a source, $\epsilon(i)=-1$ if $i$ is a sink, and $\epsilon(i) \in \{\pm1\}$ if $i$ is isolated. 
If $B$ is a block of $A$, then its valued quiver 
$\Gamma_B$ corresponds to some connected component of $\Gamma_A$.
Suppose that $1,\ldots, k$ be the set of vertices of $\Gamma_B$ ($1\leq k \leq n$),
then the restriction $\epsilon_A|_{[k]} \colon [k] \rightarrow \{\pm1\}$ gives the map defined in Section \ref{torsion class for hereditary}, namely, $\epsilon_B:=\epsilon_A|_{[k]}$. 
By Theorem \ref{hereditary case}, we have an isomorphism $\tors_{\epsilon_B} B \rightarrow \fators B^!$.

Consequently, if $A$ is decomposed into $A=B_1\times \cdots \times B_m$ with $B_i$ connected algebras, then we obtain a family of isomorphisms $R_i \colon \tors_{\epsilon_{B_i}} B_i \rightarrow \fators B_i^!$ for all $i=1,\ldots, m$ induced by $\epsilon_A$.
If we set $A^!:=B_1^! \times \cdots \times B_m^!$, then it is easy to see that $R_1,\ldots, R_m$
gives rise to a unique isomorphism of partially ordered sets
$$
R \colon \tors_{\epsilon_A} A \rightarrow \fators A^!
$$
such that $R|_{\tors B_i} = R_i$ for all $i = 1,\ldots,m$. 
Similarly, we get an isomorphism $\stau_{\epsilon_A} A \rightarrow \tilt A^!$ from Theorem \ref{stau-tilt}. 
\end{remark}

\section{Torsion classes of algebras with radical square zero} 
\label{algebras with RSZ}

In this section, we study torsion classes of an arbitrary algebra with radical square zero. 
Throughout this section, let $A$ be an algebra with radical square zero and $\mathcal{S}(A)=\{e_1, \ldots, e_n\}$ a complete set of primitive orthogonal idempotents of $A$. 
To each map $\epsilon \colon [n] \rightarrow \{\pm1\}$, we associate the following upper triangular matrix algebras: 
$$
A_{\epsilon}:=\begin{pmatrix}
e^{\epsilon,+}(A/J_A)e^{\epsilon,+} & e^{\epsilon,+} J_A e^{\epsilon,-} \\ 
0 & e^{\epsilon,-}(A/J_A) e^{\epsilon,-}
\end{pmatrix} \ and \
A_{\epsilon}^! := \begin{pmatrix} 
e^{\epsilon,-}(A/J_A)e^{\epsilon,-} & D(e^{\epsilon,+} J_A e^{\epsilon,-}) \\ 
0   &  e^{\epsilon,+}(A/J_A)e^{\epsilon,+}
\end{pmatrix}. 
$$

The following is basic observations. 
 
\begin{proposition} \label{A and A epsilon}
\begin{enumerate}
\item $A_{\epsilon}$ and $A_{\epsilon}^!$ are hereditary algebras with radical square zero. 
\item The Jacobson radical of $A_{\epsilon}$ (resp. $A_{\epsilon}^!$) is 
$e^{\epsilon,+} J_A e^{\epsilon,-}$ (resp. $D(e^{\epsilon,+} J_A e^{\epsilon,-})$).
\item $A_{\epsilon}\cong A/I_{\epsilon}$, where 
$I_{\epsilon}=e^{\epsilon,+}J_A e^{\epsilon,+} + e^{\epsilon,-}J_A e^{\epsilon,+} + 
e^{\epsilon,-} J_A e^{\epsilon,-}$ is an two-sided ideal in $A$.
\end{enumerate}
\end{proposition}

\begin{proof}
(1): It follows from \cite[Section III.2]{ARS}. 

(2): Since $A_{\epsilon}$ and $A^!_{\epsilon}$ are defined as upper triangular matrix algebra, the assertion follows from \cite[Section III.2, Proposition 2.5(a)]{ARS}.

(3): Firstly, we check that $I_{\epsilon}$ is a two-sided ideal in $A$. 
Using $J_A^2=0$, we have 
\begin{eqnarray*}
I_{\epsilon}A &=& 
e^{\epsilon,+}J_A e^{\epsilon,+}A + e^{\epsilon,-}J_A e^{\epsilon,+}A + 
e^{\epsilon,-} J_A e^{\epsilon,-}A \\ &=&
e^{\epsilon,+}(J_A e^{\epsilon,+}A)e^{\epsilon,+} + 
e^{\epsilon,-}(J_A e^{\epsilon,+}A) e^{\epsilon,+} + 
e^{\epsilon,-}(J_A e^{\epsilon,-}A) e^{\epsilon,-} \\ &\subseteq &
e^{\epsilon,+} J_A e^{\epsilon,+} + 
e^{\epsilon,-} J_A  e^{\epsilon,+} + 
e^{\epsilon,-} J_A e^{\epsilon,-} = I_{\epsilon}
\end{eqnarray*}
Therefore, we have $I_{\epsilon} A \subseteq I_{\epsilon}$. 
Similarly, we have $A I_{\epsilon} \subseteq I_{\epsilon}$, and hence $I_{\epsilon}$ is a two-sided ideal in $A$.
By comparing $A_{\epsilon}$ with 
a decomposition 
$A = e^{\epsilon,+}A e^{\epsilon,+} + e^{\epsilon,+} A e^{\epsilon,-} + e^{\epsilon,-} A e^{\epsilon,+} +e^{\epsilon,-} Ae^{\epsilon,-}$, we get the assertion.
\end{proof}

From our construction, there is a natural bijection $\mathcal{S}(A) \rightarrow \mathcal{S}(A_{\epsilon}):=\{e'_1,\ldots, e'_n\}$ mapping $e_i \mapsto e_i'$ for all $i=1,\ldots, n$.

\begin{remark} \rm \label{components}
Let $\Gamma=\Gamma_A$ be a valued quiver of $A$ and $\epsilon \colon [n] \rightarrow \{\pm1\}$ a map. 
Let $\Gamma_{\epsilon}$ be a valued quiver given by
\begin{itemize}
\item the set of vertices is the same as $\Gamma$, and 
\item the set of arrows in $\Gamma_{\epsilon}$ is 
$$\{\text{$i \overset{(d_{ij}', d_{ij}'')}{\longrightarrow} j$ in $\Gamma$} \mid \text{$\epsilon(i)=1$, $\epsilon(j)=-1$}\}.$$
\end{itemize}
Then it gives a valued quiver of $A_{\epsilon}$. 
In particular, this $\epsilon$ gives a map $\epsilon_{A_{\epsilon}}=\epsilon \colon [n] \rightarrow \{\pm1\}$ for $A_{\epsilon}$ discussed in Remark \ref{extend to non-connected} because $\epsilon(i)=1, -1$ if $i$ is a source and a sink respectively.
\end{remark}

\begin{example} \rm \label{example tri}
Let $Q$ be a finite quiver and 
$A:=kQ/I$ where $I$ is a two-sided ideal generated by all path of length $2$. 
Then $A$ is an algebra with radical square zero whose (valued) quiver is $Q$.
In this situation, we have isomorphisms $A_{\epsilon}\cong kQ_{\epsilon}$ and 
$A_{\epsilon}^! \cong kQ_{\epsilon}^{\rm op}$ for any $\epsilon \in \{\pm1\}^n$, where $Q_{\epsilon}$ is a quiver defined in Remark \ref{components}. 

Let $Q=\begin{xy}(5,3)*{1}="1", (10,-2)*{2}="2",  (0,-2)*{3}="3",
{"1" \ar "2"}, {"2" \ar "3"}, {"3" \ar "1"}, 
\end{xy}$ be a quiver. For each $\epsilon \in \{\pm1\}^3$, the quiver $Q_{\epsilon}$ is given as follows:

\begin{table}[h]
\begin{tabular}{|c|c|c|c|}
\hline 
\begin{xy}
(-10,0)*{Q_{(1,1,1)} \colon}, (5,3)*{1}="1", (10,-2)*{2}="2", (0,-2)*{3}="3", 
\end{xy} & \begin{xy}
(-10,0)*{Q_{(-1,1,1)} \colon}, (5,3)*{1}="1", (10,-2)*{2}="2", (0,-2)*{3}="3",  {"3" \ar "1"}
\end{xy} &
\begin{xy}
(-10,0)*{Q_{(1,-1,1)} \colon}, (5,3)*{1}="1", (10,-2)*{2}="2", (0,-2)*{3}="3", {"1" \ar "2"}
\end{xy} &
\begin{xy}
(-10,0)*{Q_{(1,1,-1)} \colon}, (5,3)*{1}="1", (10,-2)*{2}="2", (0,-2)*{3}="3", {"2" \ar "3"}
\end{xy} \\ \hline
\begin{xy}
(-10,0)*{Q_{(-1,-1,1)} \colon}, (5,3)*{1}="1", (10,-2)*{2}="2", (0,-2)*{3}="3", {"3" \ar "1"}
\end{xy} & 
\begin{xy}
(-10,0)*{Q_{(1,-1,-1)} \colon}, (5,3)*{1}="1", (10,-2)*{2}="2", (0,-2)*{3}="3", {"1" \ar "2"}
\end{xy}&
\begin{xy}
(-10,0)*{Q_{(-1,1,-1)} \colon}, (5,3)*{1}="1", (10,-2)*{2}="2", (0,-2)*{3}="3", {"2" \ar "3"}
\end{xy}&
\begin{xy}
(-12,0)*{Q_{(-1,-1,-1)} \colon}, (5,3)*{1}="1", (10,-2)*{2}="2", (0,-2)*{3}="3",
\end{xy} \\  \hline
\end{tabular}
\end{table}
\end{example}

\subsection{Torsion classes} 
An aim of this section is to prove the following result.

\begin{theorem} \label{tors-fators}
Let $A$ be a finite dimensional  algebra with radical square zero. 
Let $\mathcal{S}(A):=\{e_1,\ldots, e_n\}$ be a complete set of primitive orthogonal idempotents of $A$.
For each map $\epsilon \colon [n] \rightarrow \{\pm1\}$, 
there are isomorphisms of partially ordered sets between
\begin{enumerate}
\item $\tors_{\epsilon} A$,
\item $\tors_{\epsilon} A_{\epsilon}$,
\item $\fators A_{\epsilon}^!$.
\end{enumerate}

In particular, there is a bijection
$$
\tors A \ \overset{1-1}{\longleftrightarrow} \ \bigsqcup_{\epsilon \colon [n] \rightarrow \{\pm1\}} \fators A_{\epsilon}^!. 
$$
\end{theorem}

From now on, we prove Theorem \ref{tors-fators} for a fixed map $\epsilon \colon [n] \rightarrow \{\pm1\}$. 
By Proposition \ref{A and A epsilon}(3), $A_{\epsilon}\cong A/I_{\epsilon}$ where $I_{\epsilon} = e^{\epsilon,+}J_Ae^{\epsilon,+} +  e^{\epsilon,-}J_Ae^{\epsilon,+} +  e^{\epsilon,-}J_Ae^{\epsilon,-} 
$. So, $\module A_{\epsilon}$ can be regarded as a full subcategory of $\module A$. 
On the other hand, there is a full functor
$$
F_{\epsilon}:= - \otimes_{A} A_{\epsilon} \colon \module A \longrightarrow \module A_{\epsilon}.
$$

\begin{lemma} 
\label{mod-str} The following hold. 
\begin{enumerate}
\item[(a)] $P_{A_{\epsilon}}^{\epsilon,-} \cong S_{A}^{\epsilon,-}$ and $\rad P_{A_{\epsilon}}^{\epsilon,+} \in \add S_A^{\epsilon,-}$ in $\module A$.

\item[(b)] For an $A$-module $M$, we have the following commutative diagram of short exact sequences in $\module A$. 
$$\hspace{10mm}
\xymatrix{
   &    & 0 &0  &  \\
   &    & \tops M  \ar^{\sim \quad}[r] \ar[u] & \tops F_{\epsilon}(M) \ar[u] \ar[r] & 0\\
 0 \ar[r] & M\otimes_A I_{\epsilon} \ar[r] & M  \ar[r]\ar[u] & F_{\epsilon}(M) \ar[u] \ar[r] &0\\
 0  \ar[r] & M \otimes_{A} I_{\epsilon} \ar[r] \ar@{=}[u] &  \rad M \ar[r]\ar[u] & \rad F_{\epsilon}(M) \ar[r]\ar[u]  & 0 &\text{\rm (split)}\\
   &    &  0 \ar[u]  & 0 \ar[u]  &  
}
$$
In particular, 
there is an isomorphism 
$\rad M \cong \rad F_{\epsilon}(M)  \oplus (M \otimes_A I_{\epsilon})$ of semisimple modules. 
Moreover, 
\begin{itemize}
\item if $M \in \Fac P_{A}^{\epsilon,+}$, then $M\otimes_A I_{\epsilon} \in \add S_{A}^{\epsilon,+}$ and  $\rad F_{\epsilon}(M) \in \add S_A^{\epsilon,-}$, and 
\item if $M \in \Fac P_{A}^{\epsilon,-}$, then $F_{\epsilon}(M) \in \add S_A^{\epsilon,-}$. 
\end{itemize} 
\end{enumerate}
\end{lemma}

\begin{proof}
(a) It is clear by the definition of $A_{\epsilon}$ as an upper triangular matrix algebra.

(b) 
There is a commutative diagram of short exact sequence: 
$$ \hspace{15mm}
\xymatrix{
0 \ar[r] & I_{\epsilon}  \ar[r]& A  \ar[r] & A_{\epsilon} \ar[r]& 0 \\ 
0 \ar[r] & I_{\epsilon}  \ar[r] \ar@{=}[u]& \rad A  \ar[r] \ar@{^{(}->}[u]& \rad A_{\epsilon} \ar[r] \ar@{^{(}->}[u]& 0  & \text{(split)}
}
$$
where the bottom one splits since $\rad A$ is semisimple.  
Applying $M\otimes_{A}-$, we obtain the desired diagram. 

Next, let $M\in \Fac P_A^{\epsilon,+}$ and take a surjection $f\colon (P_{A}^{\epsilon,+})^{r} \rightarrow M$ for some integer $r>0$.    
Since $F_{\epsilon}$ is right exact, it induces a surjective homomorphism 
$(\rad P_{A_{\epsilon}}^{\epsilon,+})^{r} \rightarrow \rad F_{\epsilon}(M)$. 
By (a), $\rad P_{A_{\epsilon}}^{\epsilon,+} \in \add S_A^{\epsilon,-}$, 
so $\rad F_{\epsilon} (M) \in \add S^{\epsilon,-}_{A}$. 
On the other hand, we have $M\otimes_AI_{\epsilon} \in \add S_{A}^{\epsilon,+}$ since 
it is a semisimple module satisfying $(M\otimes_AI_{\epsilon})e^{\epsilon,-}=0$.
Finally, if $M \in \Fac P_{A}^{\epsilon,-}$, then $F_{\epsilon}(M) \in \Fac P_{A_{\epsilon}}^{\epsilon,-} = \add S_{A}^{\epsilon, -}$ by (a). 
\end{proof}

Now, we are ready to prove Theorem \ref{tors-fators}.

\begin{proof}[Proof of Theorem \ref{tors-fators}] \ \\ 
(1) and (2): Since $A_{\epsilon}$ is a factor algebra of $A$, we have a morphism of partially ordered sets by Proposition \ref{sign and factor}: 
\begin{equation} \label{tors-epsilon}
\overline{(-)} \colon \tors_{\epsilon} A \longrightarrow \tors_{\epsilon} A_{\epsilon}.
\end{equation}

Let $\mathcal{U} \in \tors_{\epsilon} A_{\epsilon}$ be a torsion class 
and $\mathsf{T}(\mathcal{U})$ the smallest torsion class in $\module A$ containing $\mathcal{U}$.
Clearly, $\mathsf{T}(\mathcal{U})$ lies in $\tors_{\epsilon} A$ and satisfies $\overline{\mathsf{T}(\mathcal{U})}=\mathcal{U}$. 
So the map (\ref{tors-epsilon}) is surjective. 
To prove that this is injective, it is enough to see that 
$\overline{\mathcal{T}}=\mathcal{U}$ implies $\mathcal{T}=\mathsf{T}(\mathcal{U})$
for any $\mathcal{T}\in \tors_{\epsilon} A$. 
Let $\mathcal{T} \in \tors_{\epsilon} A$ be a torsion class such that $\overline{\mathcal{T}}=\mathcal{U}$.
By definition, we have $\mathsf{T}(\mathcal{U})\subseteq \mathcal{T}$. 
Conversely, let $X \in \mathcal{T} \subseteq \Fac P_{A}^{\epsilon,+}$. 
By Lemma \ref{mod-str}(b), there is a short exact sequence
$$
0 \rightarrow X\otimes_A I_{\epsilon} \rightarrow X \rightarrow F_{\epsilon}(X) \rightarrow 0
$$
with $X\otimes_A I_{\epsilon} \in \add S_{A}^{\epsilon,+} \subseteq \mathcal{U}$   
and $F_{\epsilon}(X) \in \mathcal{T}\cap \module A_{\epsilon} = \overline{\mathcal{T}} = \mathcal{U}$. 
Since $\mathsf{T}(\mathcal{U})$ is closed under extensions over $\module A$, 
we obtain $X \in \mathsf{T}(\mathcal{U})$, and hence $\mathcal{T} = \mathsf{T}(\mathcal{U})$.
Therefore, the map (\ref{tors-epsilon}) is a bijection.

(2) and (3): Since $\epsilon$ provides the map $\epsilon_{A_{\epsilon}}$ for $A_{\epsilon}$ 
in Remark \ref{extend to non-connected}, 
there is an isomorphism of partially ordered sets $\tors_{\epsilon} A_{\epsilon} \rightarrow \fators A_{\epsilon}^!$ as in Remark \ref{extend to non-connected}. 

It completes the proof.
\end{proof}

Next, we focus on functorially finite torsion classes. 
Let $R_{\epsilon} \colon \tors_{\epsilon} A_{\epsilon} \rightarrow \fators A_{\epsilon}^!$ be an isomorphism given in Remark \ref{extend to non-connected}, 
as in the proof of Theorem \ref{tors-fators}. 
Similarly, let $\varphi_{\epsilon} \colon \stau_{\epsilon} A \rightarrow \fators A^!_{\epsilon}$ be an isomorphism. 
Our result is the following.

\begin{theorem} \label{stau-tilt-tilt}
Let $A$ be an algebra with radical square zero. 
For each map $\epsilon \colon [n] \rightarrow \{\pm1\}$, 
the following diagram is commutative, and all arrows are isomorphisms of partially ordered sets:
$$
\xymatrix@C40pt{\ar@{}[rd]|{\circlearrowright}
\stau_{\epsilon} A \ar[r]^{F_{\epsilon}} \ar[d]^{\Fac} & \stau_{\epsilon} A_{\epsilon} \ar[r]^{\varphi_{\epsilon}} \ar[d]^{\Fac} \ar@{}[rd]|{\circlearrowright}& \tilt A_{\epsilon}^!. \ar[d]^{\Fac} \\
\ftors_{\epsilon} A \ar[r]^{\overline{(-)}} & \ftors_{\epsilon} A_{\epsilon} \ar[r]^{R_{\epsilon}} & \ffators A_{\epsilon}^!. \\
}
$$
In particular, we have a bijection 
$$
\stau A \ \overset{1-1}{\longleftrightarrow} \ \bigsqcup_{\epsilon \colon [n] \rightarrow \{\pm1\}} \tilt A_{\epsilon}^!.
$$
\end{theorem}

We need the following observations.

\begin{lemma} \label{proj-str}
For any $X \in \add P_{A}^{\epsilon,+}$ and $Y \in \add P_{A}^{\epsilon,-}$, we have 
\begin{equation} \label{ef-epsilon}
\Hom_{A}(Y,X) \cong \Hom_{A_{\epsilon}} (F_{\epsilon}(Y), F_{\epsilon}(X)).
\end{equation}
\end{lemma}

\begin{proof}
It is immediate from the equality
\[
\Hom_A(P_{A}^{\epsilon,-}, P_{A}^{\epsilon,+})=e^{\epsilon,+}Ae^{\epsilon,-} =e^{\epsilon,+} A_{\epsilon}e^{\epsilon,-} =\Hom_{A_{\epsilon}}(P_{A_{\epsilon}}^{\epsilon,-},  P_{A_{\epsilon}}^{\epsilon,+}).  \qedhere
\] 
\end{proof}

\begin{lemma} \label{tau-rigid-pair}
Let $M$ be an $A$-module and 
$P^{-1} \rightarrow P^0 \overset{f}{\rightarrow} M \rightarrow 0$ a minimal projective presentation of $M$ such that $P^0 \in \add P_A^{\epsilon,+}$ and $P^{-1} \in \add P_{A}^{\epsilon,-}$. 
Then the following hold. 
\begin{enumerate}
\item For any $Q \in \add P_{A}^{\epsilon,-}$, we have $\Hom_A(Q, M)\cong \Hom_{A_{\epsilon}}(F_{\epsilon}(Q), F_{\epsilon}(M))$. 
\item  $M$ is $\tau$-rigid if and only if $F_{\epsilon}(M)$ is $\tau$-rigid.
\end{enumerate}
\end{lemma}

\begin{proof}
(1): It follows from the fact that any non-zero homomorphism $g \colon Q \rightarrow M$ factor through $\rad F_{\epsilon}(M)$ by Lemma \ref{mod-str}.
 
(2): By (1), we have the following commutative diagram:
$$
\xymatrix@C60pt{
\Hom_A(P_0, M) \ar[r]^{(f, M)} \ar@{->>}[d] \ar@{}[rd]|{\circlearrowright}& \Hom_A(P_1, M) \ar[d]_{(1)}^{\rotatebox{90}{$\sim$}} \\ 
\Hom_{A_{\epsilon}}(F_{\epsilon}(P_0), F_{\epsilon}(M)) \ar[r]^{(F_{\epsilon}(f), F_{\epsilon}(M))} & \Hom_{A_{\epsilon}}(F_{\epsilon}(P_1), F_{\epsilon}(M)). 
}
$$
By Proposition \ref{char-tau}, $M$ is $\tau$-rigid if and only if $(f,M)$ is surjective. 
From the above diagram, this is equivalent to the condition that 
$(F_{\epsilon}(f), F_{\epsilon}(M))$ is surjective, and also that $F_{\epsilon}(M)$ is $\tau$-rigid by Proposition \ref{char-tau} again.
\end{proof}

Now, we are ready to prove Theorem \ref{stau-tilt-tilt}.

\begin{proof}[Proof of Theorem \ref{stau-tilt-tilt}]
By Lemma \ref{proj-str}, $F_{\epsilon}=-\otimes_A A_{\epsilon}$ gives a bijection between 
isomorphism classes of indecomposable modules in $\module A$ and in $\module A_{\epsilon}$ 
whose $g$-vectors lie in $\{ x \in \mathbb{Z}^n \mid \text{$x_i \in \epsilon(i) \cdot \mathbb{Z}_{>0}$ for all $i\in [n]$}\}$. 
By Proposition \ref{char-tau}(2) and Lemma \ref{tau-rigid-pair}, 
$F_{\epsilon}$ induces an isomorphism $\stau_{\epsilon} A \overset{\sim}{\rightarrow} \stau_{\epsilon} A_{\epsilon}$.  
Commutativity of the left square follows from \cite[Proposition 5.6(b)]{DIRRT}.
On the other hand, by Proposition \ref{stau-tilt}, the map $\varphi_{\epsilon}$ commutes the right square. 
We finish the proof.  
\end{proof}

We can also extend our calculation on $g$-vectors. 
The following result means that the diagonal matrix $\mathrm{B}_{\epsilon}:={\rm diag}(\epsilon(1), \ldots, \epsilon(n))$ gives a transformation between $g$-vectors and dimension vectors.

\begin{theorem} \label{transformation}
For each map $\epsilon \colon [n] \rightarrow \{\pm1\}$, we have the following commutative diagram:
$$
\xymatrix@C40pt{
\stau_{\epsilon} A  \ar[r]^{F_{\epsilon}} \ar[d]^{g_{A}} \ar@{}[rd]|{\circlearrowright}&\stau_{\epsilon} A_{\epsilon} \ar[r]^{\varphi_{\epsilon}} \ar[d]^{g_{A_{\epsilon}}}  \ar@{}[rd]|{\circlearrowright}& \tilt A_{\epsilon}^{!} 
\ar[d]^{c_{A_{\epsilon}^{!}}} \\ 
\mathbb{Z}^n \ar[r]^{{\rm I}_n} &  \mathbb{Z}^n \ar[r]^{{\rm B}_\epsilon}& \mathbb{Z}^n
}
$$
\end{theorem}

\begin{proof}
It is immediate from Proposition \ref{g-c-vector} and Theorem \ref{stau-tilt-tilt}.
\end{proof} 

\begin{example}  \rm \label{example epsilon}
Let $Q: \begin{xy}(5,3)*{1}="1", (10,-2)*{2}="2",  (0,-2)*{3}="3",
{"1" \ar "2"}, {"2" \ar "3"}, {"3" \ar "1"}, 
\end{xy}$ be a quiver and $A:=kQ/I$ where $I$ is a two-sided ideal generated by all path of length $2$. 
Take an element $\epsilon=(1,-1,1) \in \{\pm1\}^3$, we have 
$A_{\epsilon}=kQ_{(1,-1,1)}$ and $A^!_{\epsilon} = kQ_{(1,-1,1)}^{\rm op}$ as in Example \ref{example tri}.
By Theorem \ref{stau-tilt-tilt}, 
there is the following isomorphisms of the Hasse quivers, where we describe a given module by its composition series. 

\begin{table}[h] 
  \begin{tabular}{|c|c|c|c}
\hline $\Hasse(\stau_{\epsilon} A)$ & $\Hasse(\stau_{\epsilon} A_{\epsilon})$ & $\Hasse(\tilt A_{\epsilon}^!)$ \\ \hline \hline
\big(\scalebox{0.8}{$\begin{xy}(0,2)*{1},(0,-1)*{2} \end{xy} \oplus \begin{xy}(0,0)*{1} \end{xy} \oplus \begin{xy}(0,2)*{3},(0,-1)*{1} \end{xy}$} \big)  & 
\big(\scalebox{0.8}{$\begin{xy}(0,2)*{1},(0,-1)*{2} \end{xy} \oplus \begin{xy}(0,0)*{1} \end{xy} \oplus \begin{xy}(0,0)*{3} \end{xy}$} \big)  & 
\big(\scalebox{0.8}{$\begin{xy}(0,0)*{1} \end{xy} \oplus \begin{xy}(0,2)*{2}, (0,-1)*{1} \end{xy} \oplus \begin{xy}(0,0)*{3} \end{xy}$} \big)  \\ 
\rotatebox{90}{$\longleftarrow$} & \rotatebox{90}{$\longleftarrow$} &\rotatebox{90}{$\longleftarrow$} \\ 
\big(\scalebox{0.8}{$ \begin{xy}(0,0)*{1} \end{xy} \oplus \begin{xy}(0,2)*{3},(0,-1)*{1} \end{xy}$} \big)  & 
\big(\scalebox{0.8}{$ \begin{xy}(0,0)*{1} \end{xy} \oplus \begin{xy}(0,0)*{3} \end{xy}$} \big)  & 
\big(\scalebox{0.8}{$ \begin{xy}(0,0)*{2} \end{xy} \oplus \begin{xy}(0,2)*{2}, (0,-1)*{1} \end{xy} \oplus \begin{xy}(0,0)*{3} \end{xy}$} \big)  \\ \hline
\end{tabular}
 \end{table}

Considering all $\epsilon \in \{\pm1\}^3$, we get a disjoint union of $\Hasse (\stau_{\epsilon} A) \subset \Hasse(\stau A)$, whose arrows are described as a real line in the following figure:

$$
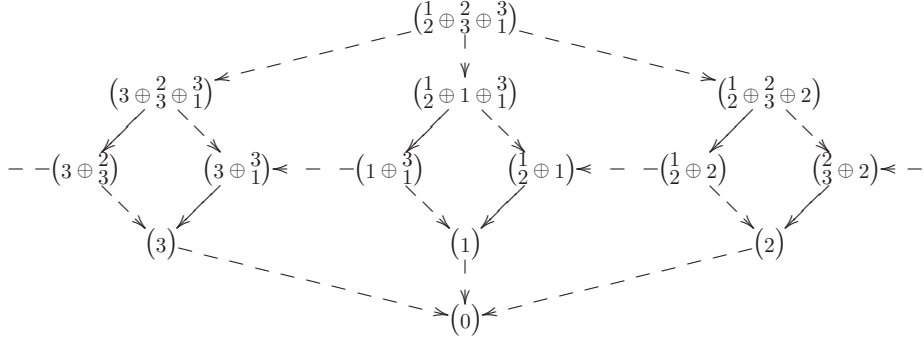
\begin{xy}
(0,0)*{\big(\scalebox{0.8}{$ \begin{xy}(0,2)*{1}, (0,-1)*{2}  \end{xy} \oplus \begin{xy}(0,2)*{2},(0,-1)*{3} \end{xy} \oplus  \begin{xy}(0,2)*{3}, (0,-1)*{1}  \end{xy} $} \big)}="11",
(-40, -10)*{\big(\scalebox{0.8}{$ \begin{xy} (0,0)*{3}  \end{xy} \oplus \begin{xy}(0,2)*{2},(0,-1)*{3} \end{xy} \oplus  \begin{xy}(0,2)*{3}, (0,-1)*{1}  \end{xy} $} \big)}="21",
(-50,-20)*{\big(\scalebox{0.8}{$ \begin{xy}(0,0)*{3}  \end{xy} \oplus \begin{xy}(0,2)*{2},(0,-1)*{3} \end{xy}  $} \big)}="22",
{"21" \ar "22"},
(0,-10)*{\big(\scalebox{0.8}{$ \begin{xy} (0,2)*{1}, (0,-1)*{2}  \end{xy} \oplus \begin{xy}(0,0)*{1} \end{xy} \oplus  \begin{xy}(0,2)*{3}, (0,-1)*{1}  \end{xy} $} \big)}="31",
(-10, -20)*{\big(\scalebox{0.8}{$ \begin{xy}(0,0)*{1}  \end{xy} \oplus \begin{xy}(0,2)*{3},(0,-1)*{1} \end{xy}  $} \big)}="32",
{"31" \ar "32"},
(40,-10)*{\big(\scalebox{0.8}{$ \begin{xy} (0,2)*{1}, (0,-1)*{2}  \end{xy} \oplus \begin{xy}(0,2)*{2},(0,-1)*{3} \end{xy} \oplus  \begin{xy} (0,0)*{2}  \end{xy} $} \big)}="41",
(30,-20)*{\big(\scalebox{0.8}{$ \begin{xy}(0,2)*{1}, (0,-1)*{2}  \end{xy} \oplus \begin{xy}(0,0)*{2} \end{xy}  $} \big)}="42",
{"41" \ar "42"},
(-30, -20)*{\big(\scalebox{0.8}{$ \begin{xy} (0,0)*{3}  \end{xy} \oplus \begin{xy}(0,2)*{3}, (0,-1)*{1}  \end{xy} $} \big)}="51",
(-40,-30)*{\big(\scalebox{0.8}{$ \begin{xy}(0,0)*{3}  \end{xy}  $} \big)}="52",
{"51" \ar "52"},
(10,-20)*{\big(\scalebox{0.8}{$ \begin{xy} (0,2)*{1}, (0,-1)*{2}  \end{xy} \oplus \begin{xy}(0,0)*{1} \end{xy} $} \big)}="61",
(0, -30)*{\big(\scalebox{0.8}{$ \begin{xy}(0,0)*{1}  \end{xy}$} \big)}="62",
{"61" \ar "62"},
(50,-20)*{\big(\scalebox{0.8}{$ \begin{xy}(0,2)*{2},(0,-1)*{3} \end{xy} \oplus  \begin{xy} (0,0)*{2}  \end{xy} $} \big)}="71",
(40,-30)*{\big(\scalebox{0.8}{$ \begin{xy}(0,0)*{2} \end{xy}  $} \big)}="72",
{"71" \ar "72"},
(0,-40)*{\big(\scalebox{0.8}{$ \begin{xy}(0,0)*{0} \end{xy}$} \big)}="81",
{"11" \ar@{-->} "21"}, {"11" \ar@{-->} "31"}, {"11" \ar@{-->} "41"},
{"21" \ar@{-->} "51"}, {"31" \ar@{-->} "61"}, {"41" \ar@{-->} "71"},
{"22" \ar@{-->} "52"}, {"32" \ar@{-->} "62"}, {"42" \ar@{-->} "72"},
{"22" \ar@{--} (-60, -20)}, {"32" \ar@{-->} "51"}, {"42" \ar@{-->} "61"}, {(60, -20) \ar@{-->} "71"},
{"52" \ar@{-->} "81"}, {"62" \ar@{-->} "81"}, {"72" \ar@{-->} "81"},
\end{xy}
$$

As wee see, dotted arrows in $\Hasse (\stau A)$ are missing,   
which connect components having different signatures.  
In the next section, we will reconstruct such arrows 
and give a complete description of $\Hasse (\stau A)$.
\end{example}

We end this subsection with a result on silting theory. 
It is important to know whether a given silting complex $T$ is {\it tilting}, 
that is, satisfying the additional condition $\Hom_A(T,T[i])=0$ for any integer $i\neq 0$.
For the algebra $A$ with radical square zero, we obtain the following characterization of 
two-term tilting complexes.

\begin{proposition} 
For a given map $\epsilon \colon [n] \rightarrow \{\pm1\}$, the following conditions are equivalent.  
\begin{itemize}
\item[(1)] $e^{\epsilon,-}Ae^{\epsilon,+}=0$,
\item[(2)] Any complex in $\twosilt_{\epsilon} A$ is tilting,  
\item[(3)] There exists a tilting complex in $\twosilt_{\epsilon} A$.
\end{itemize}
\end{proposition}

\begin{proof}
Let $T=(T^{-1} \overset{d_T}{\rightarrow} T^{0})$ be a basic two-term silting complex in $\twosilt_{\epsilon} A$. By definition, $T^0 \in \add P_{A^!}^{\epsilon,+}$ and $T^{-1} \in \add P_{A}^{\epsilon,-}$. 
In this case, we have $\Hom_{\Kb(\proj A)}(T, T[-1]) \cong \Hom_{A}(T^{0}, T^{-1})$.  
In fact, we have $f \circ d_T = 0 = d_T\circ f$ for any $f \in \Hom_{A}(T^{0}, T^{-1})$ since 
both of $d_T$ and $f$ lie in the radical of $\proj A$. 
Therefore, $T$ is tilting if and only if $\Hom_{\Kb(\proj A)}(T^{0}, T^{-1})=0$ by definition.   
On the other hand, there is a natural isomorphism $e^{\epsilon,-}Ae^{\epsilon,+}\cong \Hom_A(P_A^{\epsilon,+}, P_{A}^{\epsilon,-})$.

$(1) \Rightarrow (2)$ 
Assume (1). Then $T$ is a tilting complex from the above discussion. 

$(2) \Rightarrow (3)$ 
It is clear since $\twosilt_{\epsilon} A \cong \tilt A_{\epsilon}^!$ is non-empty. 

$(3) \Rightarrow (1)$ 
We show the contraposition. 
Assume that there exists a non-zero homomorphism $f \in \Hom_{A}(P_{A}^{\epsilon,+}, P_{A}^{\epsilon,-})$. 
By Proposition \ref{properties of two-silt}(4), there is a direct summand of $T^{0}$ (resp. $T^{-1}$) which is isomorphic to $P_{A}^{\epsilon,+}$(resp. $P_{A}^{\epsilon,-}$).
So, $f$ gives a non-zero homomorphism in $\Hom_{A}(T^{0}, T^{-1})$. 
From the above discussion, it implies that $T$ is not tilting. 
\[
\begin{gathered}[b]
  \xymatrix{ 
    T^{-1} \ar[r]^{d_T} \ar[d] & T^{0} \ar_{f}[d]  \ar[r]    & \ar[d] 0  \\
0 \ar[r]  & T^{-1}  \ar[r]^{d_T} & T^0}\\
\end{gathered}
\qedhere
\] 
\end{proof}

\subsection{Gluing the Hasse quivers} 

The bijection in Theorem \ref{stau-tilt-tilt} provides a partial data of the Hasse quiver of $\stau A$, indeed, some arrows are missing (see Example \ref{example epsilon}). 
In contrast, we can reconstruct it in terms of tilting mutation (Theorem \ref{Hasse}).  

\begin{proposition} \cite[Theorem 2.18]{AIR}  \rm 
\label{mutation} 
Up to isomorphism, any basic almost complete $\tau$-tilting pair $(U,P)$ for $A$ is a direct summand of 
precisely two basic $\tau$-tilting pairs $(M, Q), (N, Q')$ for $A$, 
that is, $U$ (respectively, $P$) is a direct summand of $M, N$ (respectively, $Q,Q'$). 
In this case, either $N<M$ or $N>M$ holds in $\stau A$.
\end{proposition}

In the situation of Proposition \ref{mutation}, we say that $N$ is a {\it left mutation} (respectively,{\it right mutation}) of $M$ when 
 $N < M$ (respectively, $N > M$). 
By \cite[Corollary 2.34]{AIR}, the Hasse quiver $\mathsf{Hasse}(\stau A)$ of $\stau A$ is given by
\begin{itemize}
\item the set of vertices is $\stau A$, 
\item there is an arrow $M \rightarrow N$ if and only if $N$ is a left mutation of $M$.  
\end{itemize}

Now, we define a partial order on $\{\pm1\}^n$ by $\eta \leq \epsilon$ for $\epsilon,\eta \in \{\pm1\}^n$ if $\eta(i) \leq \epsilon(i)$ for all $i \in [n]$. 
On this partial order, there is an arrow $a \colon \epsilon \rightarrow \eta$ in 
$\Hasse \{\pm1\}^n$ if and only if there exists an element $i_{a} \in [n]$ such that 
$\epsilon(e_{i_a})=1$, $\eta(e_{i_a})=-1$ and $\epsilon(e_j)=\eta(e_j)$ for all $j \neq i_a$.

For each arrow $a \colon \epsilon \rightarrow \eta$ in $\Hasse \{\pm1\}^n$, 
there is an natural isomorphism 
$$A_{\epsilon}^!/ \langle e'_{i_{a}} \rangle \cong A_{\eta}^!/ \langle e''_{i_{a}} \rangle$$
where $e'_{i_a}$ and $e''_{i_a}$ is an idempotent of $A^!_{\epsilon}$ and $A^!_{\eta}$ corresponding to $i_a$ respectively. 
We denote this factor algebra by $A_{a}^!$. 
We have two kinds of embedding of module categories:

$$
\module A_{a}^{!} \rightarrow \module A_{\epsilon}^{!}  \quad \text{and}  \quad 
\module A_{a}^{!} \rightarrow \module A_{\eta}^{!}. 
$$

\begin{def-prop} 
\label{tilt-comp}
Let $a \colon \epsilon_1 \rightarrow \epsilon_2$ be an arrow in $\Hasse \{\pm1\}^n$. 
Let $U \in \tilt A_a^!$. 
For each $i=1,2$, $U$ is a support $\tau$-tilting $A_{\epsilon_i}^!$-module. 
Moreover, there exists a unique element $T_i \in \tilt A_{\epsilon_i}^{!}$ such that $U$ is a left mutation of $T_i$. 
In this case, we draw an arrow $a_U \colon T_1 \rightarrow T_2$. 
\end{def-prop}

\begin{proof}
Firstly, a pair $(U, e'_{i_a}A_{\epsilon_1}^!)$ is a support $\tau$-tilting pair for $A_{\epsilon_1}^!$. 
In fact, we have $\Hom_{A_{\epsilon_1}^!}(e'_{i_a}A_{\epsilon_1}^!, U)= Ue'_{i_a}= 0$ and 
$|U|+ |e'_{i_a}A_{\epsilon_1}^!|= (n-1)+1=n$. 
And a pair $(U,0)$ is an almost complete $\tau$-tilting pair for $A^!_{\epsilon_1}$,
which is a direct summand of $(U, e'_{i_a} A_{\epsilon_1}^!)$.
By Proposition \ref{mutation}, another complement of $(U,0)$ must be of the form $(T_1, 0)$, where $T_1$ is a tilting $A^!_{\epsilon_1}$-module $T_1$ satisfying $U<T_1$.  
It means that $U$ is a left mutation of $T_1$. 
The same discussion holds for $\epsilon_2$.
\end{proof}

We get a complete description of the Hasse quiver of $\stau A$.

\begin{theorem} \label{Hasse}
There is an isomorphism of quivers  
$\mathsf{Hasse}(\stau A) \cong \mathrm{Q}(A)$,  
where $\mathrm{Q}(A)$ consists of a disjoint union of quivers $\Hasse(\tilt A_{\epsilon}^{!})$ 
for all $\epsilon \in \{\pm1\}^n$
with additional arrows $a_U$ in Definition \ref{tilt-comp}.
\end{theorem}

\begin{proof}
By Theorem \ref{stau-tilt-tilt}, there are isomorphisms of 
full subquivers $\Hasse (\stau_{\epsilon} A)$ and $\Hasse (\tilt A_{\epsilon}^{!})$.
Thus, it is enough to show a bijection between arrows not lie in those subquivers. 

Let $a \colon \epsilon_1 \rightarrow \epsilon_2$ be an arrow in $\Hasse \{\pm1\}^n$ and $U \in \tilt A_a^!$.  
Let $a_U \colon T_1 \rightarrow T_2$ be an arrow in Definition \ref{tilt-comp}. 
For $i=1,2$, we denote by $M_i \in \stau_{\epsilon_i} A$ the support $\tau$-tilting module corresponding to $T_i$ under the bijection in Theorem \ref{stau-tilt-tilt}. 
Then $M_2$ is a left mutation of $M_1$. 
In fact, they have an almost complete support $\tau$-tilting module $M_U$ as a common direct summand and satisfy $\Fac M_2 \subseteq \Fac M_1$ by a choice of signatures such that 
$\epsilon_1(i_a)=1, \epsilon_2(i_a)=-1$ and $\epsilon_1(j)=\epsilon_2(j)$ for all $j\neq i_a$.  
Thus, there is an arrow $\overline{a_U} \colon M_1 \rightarrow M_2$ in $\Hasse (\stau A)$.

On the other hand, let $f \colon M_1 \rightarrow M_2$ be an arrow which not lie in 
$\Hasse(\stau_{\epsilon} A)$ for any $\epsilon \in \{\pm1\}^n$. 
It is easy to see that there is a unique arrow $a \colon \epsilon_1 \rightarrow \epsilon_2$ in 
$\Hasse \{\pm1\}^n$ such that $M_i \in \stau_{\epsilon_i} A$ for $i=1,2$. 
Following the previous paragraph in reverse, we find that the correspondence 
$a_U \mapsto \overline{a_{U}}$ is bijective. 
Finally, we get an isomorphism $\Hasse (\stau A) \cong {\rm Q}(A)$. 
\end{proof}

\begin{example} \rm
Let $Q$ be the quiver in Example \ref{example epsilon}.
Let $a \colon \epsilon \rightarrow \eta$ be an arrow in $\Hasse \{\pm1\}^n$ such that 
$\epsilon=(1,-1,1), \eta=(-1,-1,1) \in \{\pm1\}^3$.
Since $A_{a}^!\cong k(2 \quad 3)$ is semisimple, 
there is a unique tilting $A_a^!$-module $U:=A_a^!$ itself. 
We have an arrow $a_{U}$ in Definition \ref{tilt-comp}, and 
it gives rise to an arrow $\overline{a_{U}}$ of $\Hasse (\stau A)$. 

$$
\begin{xy}
{(-15, 15) \ar@{-} (125, 15)},
{(-15, -28) \ar@{-} (125, -28)},
{(-15, 15) \ar@{-} (-15, -28)},
{(55, 15) \ar@{-} (55, -28)},
{(54, 15) \ar@{-} (54, -28)},
{(125, 15) \ar@{-} (125, -28)},
(0,10)*{\Hasse (\tilt A_{\epsilon}^!)}, (40,10)*{\Hasse (\tilt A_{\eta}^!)},
(0,0)*{\big(\scalebox{0.8}{$ \begin{xy} (0,0)*{1} \end{xy} \oplus \begin{xy}(0,2)*{2}, (0,-1)*{1} \end{xy} \oplus  \begin{xy}(0,0)*{3}\end{xy} $} \big)}="31",
(0,-12)*{\big(\scalebox{0.8}{$ \begin{xy} (0,0)*{2} \end{xy} \oplus \begin{xy}(0,2)*{2}, (0,-1)*{1} \end{xy} \oplus  \begin{xy}(0,0)*{3}\end{xy} $} \big)}="32",
{"31" \ar "32"},
(20,-22)*+{U=A_a^!}="1",
(40,-10)*{\big(\scalebox{0.8}{$ \begin{xy} (0,2)*{1}, (0,-1)*{3}  \end{xy} \oplus \begin{xy}(0,0)*{2} \end{xy} \oplus  \begin{xy}(0,0)*{3} \end{xy} $} \big)}="41",
(40,-22)*{\big(\scalebox{0.8}{$ \begin{xy} (0,2)*{1}, (0,-1)*{3}  \end{xy} \oplus \begin{xy}(0,0)*{2} \end{xy} \oplus  \begin{xy}(0,0)*{1} \end{xy} $} \big)}="42",
{"41" \ar "42"},
{"32" \ar@{.>}_{\rm left} "1"}, 
{"41" \ar@{.>}^{\rm left} "1"},
{"32" \ar@{-->}^{a_{U}} "41"},
(70,10)*{\Hasse (\stau_{\epsilon} A)}, (110,10)*{\Hasse (\stau_{\eta} A)},
(70,0)*{\big(\scalebox{0.8}{$ \begin{xy} (0,2)*{1}, (0,-1)*{2}  \end{xy} \oplus \begin{xy}(0,0)*{1} \end{xy} \oplus  \begin{xy}(0,2)*{3}, (0,-1)*{1}  \end{xy} $} \big)}="61",
(70, -12)*{\big(\scalebox{0.8}{$ \begin{xy}(0,0)*{1}  \end{xy} \oplus \begin{xy}(0,2)*{3},(0,-1)*{1} \end{xy}  $} \big)}="62",
{"61" \ar "62"},
(110, -10)*{\big(\scalebox{0.8}{$ \begin{xy} (0,0)*{3}  \end{xy} \oplus \begin{xy}(0,2)*{3}, (0,-1)*{1}  \end{xy} $} \big)}="51",
(110,-22)*{\big(\scalebox{0.8}{$ \begin{xy}(0,0)*{3}  \end{xy}  $} \big)}="52",
{"51" \ar "52"},
{"62" \ar@{-->}^{\overline{a_{U}}} "51"},
\end{xy}
$$

Considering all such pair $(a,U)$, 
we obtain all dotted arrows in Example \ref{example epsilon} and 
they give rise to the Hasse quiver $\Hasse (\stau A)$ by Theorem \ref{Hasse}. 
\end{example}

\section{Application to $\tau$-tilting-finite algebras} \label{tau-tilting-finiteness}

In this section, we study $\tau$-tilting-finite algebras. 
Here, we say that an algebra $A$ is {\it $\tau$-tilting-finite} if $\#\stau A$ is finite.

\subsection{Characterizing $\tau$-tilting-finite algebras with radical square zero}
Let $A$ be an algebra with radical square zero.
Let $\mathcal{S}(A):=\{e_1,\ldots, e_n\}$.
\begin{definition}\rm 
Let $\Gamma_A$ be a valued quiver of $A$.  
We define a new valued quiver $\Gamma^s_A$, called {\it separated quiver} as follows: 
\begin{itemize}
\item the set of vartices is $\{i^{\pm} \mid i \in[n]\}$ and 
\item For each valued arrow $i \overset{(d_{ij}',d_{ij}'')}{\longrightarrow} j$ in $\Gamma_A$, we 
draw a valued arrow $i^{+} \overset{(d_{ij}',d_{ij}'')}{\longrightarrow} j^{-}$ in $\Gamma^s_A$.
\end{itemize}

A full subquiver of $\Gamma^s_A$ is called {\it single subquiver} if it contains either $i^{+}$ or $i^{-}$ for each vertex $i$ of $\Gamma_A$.
\end{definition}

\begin{remark}\rm \label{Q-epsilon}
Let $\epsilon \colon [n] \rightarrow \{\pm1\}$. 
In Remark \ref{components}, the valued quiver $\Gamma_{A_{\epsilon}}$ of $A_{\epsilon}$ is described.  
We can regard $\Gamma_{A_{\epsilon}}$ as a maximal single subquiver of $\Gamma^s_A$ by 
identifying vertex $i \in \epsilon^{-1}(\sigma)$ in $\Gamma_{A_{\epsilon}}$ with 
$i^{\sigma}$ in $\Gamma^s_A$.
\end{remark}

As a corollary of Theorem \ref{stau-tilt-tilt}, we obtain a slight generalization of the result of  
Adachi \cite[Theorem 3.1]{Ada1} 
to a non-algebraically closed field.

\begin{corollary} 
\label{tau-tilt-finite}
Let $A$ be an algebra with radical square zero and let $\Gamma_A$ be a valued quiver of $A$.
These following are equivalent.
\begin{enumerate}
\item $A$ is $\tau$-tilting-finite.
\item The underlying valued graph of any single subquiver of $\Gamma^s_A$ is a disjoint union of Dynkin diagrams $\mathbb{A}_n$, $\mathbb{B}_n$, $\mathbb{C}_n$, $\mathbb{D}_n$, 
$\mathbb{E}_6$, $\mathbb{E}_7$, $\mathbb{E}_8$,
$\mathbb{F}_4$ and $\mathbb{G}_2$. 
\item For every map $\epsilon \colon [n] \rightarrow \{\pm1\}$, 
the underlying valued graph of $\Gamma_{A_{\epsilon}}$ is a disjoint union of Dynkin diagrams 
$\mathbb{A}_n$, $\mathbb{B}_n$, $\mathbb{C}_n$, $\mathbb{D}_n$, 
$\mathbb{E}_6$, $\mathbb{E}_7$, $\mathbb{E}_8$,
$\mathbb{F}_4$ and $\mathbb{G}_2$. 
\end{enumerate}
\end{corollary}

\begin{proof}
(1) $\Leftrightarrow$ (3): 
By Theorem \ref{stau-tilt-tilt}, 
$A$ is $\tau$-tilting-finite if and only if $\tilt A_{\epsilon}^!$ is finite for every map 
$\epsilon \colon [n] \rightarrow \{\pm1\}$. 
By Proposition \ref{tilting number}, it is equivalent to the condition that 
the underlying valued graph of $\Gamma_{A_{\epsilon}}$ is a disjoint union of 
Dynkin diagrams for every $\epsilon$.

(2) $\Leftrightarrow$ (3)
By Remark \ref{Q-epsilon}, it is clear since any subquiver of a Dynkin diagram is Dynkin.
\end{proof}

\subsection{Brauer line algebras and Brauer cycle algebras}
Combining Theorem \ref{tors-fators} and \ref{stau-tilt-tilt} with the next result, 
we can also compute torsion classes and support $\tau$-tilting modules for certain classes 
of algebras. In particular, it contains symmetric algebras with radical cube zero, which are extensively studied by \cite{CL, HK, Sei}. 
The following result is firstly observed by \cite{EJR} for $\tau$-tilting modules. 

\begin{proposition} \cite[Theorem 4.20]{DIRRT} \label{reduction}
Let $B$ be a finite dimensional $k$-algebra. 
Let $J$ be an ideal generated by central elements and contained in $\rad B$. 
Then there are isomorphisms of partially ordered sets
$$
\tors B \overset{\sim}{\rightarrow} \tors (B/J), \quad \stau B \overset{\sim}{\rightarrow} \stau (B/J).
$$
\end{proposition}

In fact, if $B$ is a weakly symmetric algebras with radical cube zero, then 
$A:=B/\soc B$ is an algebra with radical square zero, and 
$\soc B \subset \rad B$ is an ideal in $B$ satisfying the assumption of Proposition \ref{reduction}.
Therefore, we obtain isomorphisms 
$\tors_{\epsilon} B \cong \tors_{\epsilon} A \cong \fators_{\epsilon} A_{\epsilon}^!$ for every $\epsilon \in \{\pm\}^{n}$ by Theorem \ref{tors-fators}. Similarly, we have $\stau_{\epsilon} B \cong \stau_{\epsilon} A \cong \tilt A_{\epsilon}^!$.

We end this paper with enumerating the number of support $\tau$-tilting modules over 
a special class of symmetric algebras, called Brauer line algebras and Brauer cycle algebras, as an application of Theorem \ref{stau-tilt-tilt}. 

\begin{definition}
A {\it Brauer line algebra} having $n$ edges is defined by the bounded quiver algebra $A_{\mathbb{L}_n}:=k\Gamma/I$, where $\Gamma$ is a quiver given by 
$$
\begin{xy}
(-16,0)*{\Gamma \colon}, (0, 0)*+{1}="1", (18, 0)*++{2}="2",  
(36, 0)*++{}="3",
(42,0)*+{\cdots}, (48, 0)*++{}="n-2",
(68, 0)*++{n-1}="n-1", (88, 0)*+{n}="n",
{"1" \ar@(lu, ld)_{a_0} "1"},
{"1" \ar@<-1mm>_{a_1} "2"},
{"2" \ar@<-1mm>_{b_1} "1"},
{"2" \ar@<-1mm>_{a_2} "3"},
{"3" \ar@<-1mm>_{b_2} "2"},
{"n-2" \ar@<-1mm>_{a_{n-2}} "n-1"},
{"n-1" \ar@<-1mm>_{b_{n-2}} "n-2"},
{"n-1" \ar@<-1mm>_{a_{n-1}} "n"},
{"n" \ar@<-1mm>_{b_{n-1}} "n-1"},
{"n" \ar@(rd, ru)_{b_n} "n"},
\end{xy}
$$
and $I$ is an ideal in $k\Gamma$ generated by all elements
\begin{itemize}
\item $a_0a_1, a_1a_2, \ldots, a_{n-2}a_{n-1}, a_{n-1}b_n$, 
\item $b_1a_0, b_2b_1, \ldots, b_{n-1}b_{n-2}, b_nb_{n-1}$, 
\item $a_0^{m(0)}-(a_1b_1)^{m(1)}, (b_1a_1)^{m(1)}-(a_2b_2)^{m(2)}, \ldots, (b_{n-1}a_{n-1})^{m(n-1)}-b_n^{m(n)}$, 
\end{itemize}
where $m \colon \{0,\ldots,n\} \rightarrow \mathbb{Z}_{>0}$ is a function such that there is at most one element $j \in \{0,\ldots, n\}$ satisfying $m(j)>1$.
\end{definition}

For Brauer line algebras, we have the following.

\begin{theorem} \label{line}
Let $A_{\mathbb{L}_n}$ be a Brauer line algebra having $n$ edges. 
Then we have
$$
\#\stau A_{\mathbb{L}_n} = \binom{2n}{n}.
$$
\end{theorem}

Before the proof of Theorem \ref{line}, we need a few equations about 
Catalan numbers $C_n=\frac{1}{n+1}\binom{2n}{n}$. 

\begin{lemma} \cite{Sta}  \label{Catalan}
For any positive integer $n$, the following equations hold. 
\begin{enumerate}
\item $C_{n+1} = \sum_{k=0}^n C_k C_{n-k}$. 
\item $(n+2) C_{n+1} = 2(2n+1)C_n$. 
\item $\sum_{t=0}^n \binom{2t}{t}\binom{2(n-t)}{n-t} = 4^n$.
\end{enumerate}
\end{lemma}

Now, we are ready to prove Theorem \ref{line}. 

\begin{proof}[Proof of Theorem \ref{line}]
Let $A_{\mathbb{L}_n}=k\Gamma/I$ be a Brauer line algebra having $n$ edges as above.  
Let $A_n:=k \Gamma/I'$ where $I'$ is an ideal in $k\Gamma$ generated by all path of length $2$. Then $A_n$ is an algebra with radical square zero. 
By \cite[Corollary 3]{EJR}, there exists an ideal $J$ in $A_{\mathbb{L}_n}$ such that 
$A_{\mathbb{L}_n} /J \cong A_n$ and satisfies the assumption of Proposition \ref{reduction}. 
So there is an isomorphism $\stau A_{\mathbb{L}_n} \cong \stau A_n$ by Proposition \ref{reduction}. 
In particular, 
they have the same number of elements.
So, in the following, we will calculate the number $\# \stau A_n$ instead.

For each map $\epsilon \colon [n] \rightarrow \{\pm1\}$,  
the algebra $(A_n)_{\epsilon}$ is naturally isomorphic to the path algebra $k\Gamma_{\epsilon}$,
where $\Gamma_{\epsilon}$ is a quiver given in Remark \ref{components}.
It is easy to see that $\Gamma_{\epsilon}$ is a disjoint union of quiver of type $\mathbb{A}$. 
So, $A_n$ is $\tau$-tilting-finite by Corollary \ref{tau-tilt-finite}. 
In this situation, the number $\# \stau A_n$ is given by 
$$
\# \stau A_n = \sum_{\epsilon \colon [n] \rightarrow \{\pm1\}} \# \tilt k\Gamma_{\epsilon} 
$$
by Theorem \ref{stau-tilt-tilt}.
If we consider a map $-\epsilon \colon [n] \rightarrow \{\pm1\}$ given by $(-\epsilon)(i) = -\epsilon(i)$ for all $i\in [n]$, then the associated quiver $\Gamma_{-\epsilon}$ is isomorphic to the opposite quiver of $\Gamma_{\epsilon}$. 
By Proposition \ref{tilting number},  $\# \tilt k\Gamma_{-\epsilon} = \# \tilt k\Gamma_{\epsilon}^{\rm op}=\# \tilt k\Gamma_{\epsilon}$ holds.
Consequently, we have 
\begin{equation} \label{stau compute}
\# \stau A_n = \sum_{\epsilon \colon [n] \rightarrow \{\pm1\}} \# \tilt k\Gamma_{\epsilon} = 2 \sum_{\substack{\epsilon \colon [n] \rightarrow \{\pm1\} \\ \epsilon(1)=1}} \# \tilt k\Gamma_{\epsilon} .
\end{equation}

Next, we determine the shape of the quiver $\Gamma_{\epsilon}$ for a given map 
$\epsilon \colon [n] \rightarrow \{\pm1\}$ with $\epsilon(1)=1$.
Let $\epsilon \colon [n] \rightarrow \{\pm1\}$ be a map such that $\epsilon(1)=1$. 
Then it defines a sequence $(b_1,\ldots, b_r)$ by the equality
$$
\{1 \leq i \leq n\mid  \epsilon(i)=\epsilon(i+1)\} = \{b_1, b_1+b_2, \ldots, \sum_{s=1}^r b_s=n\}. 
$$
In this situation, $\Gamma_{\epsilon}$ is a disjoint union of $r$ connected quivers $Q_1,\ldots, Q_r$ of type $\mathbb{A}$ having $b_1,\ldots, b_r$ vertices respectively, and hence we have 
\begin{equation} \label{decomposition into type A}
\# \tilt k\Gamma_{\epsilon} = \prod_{s=1}^r \#\tilt  kQ_{s} = \prod_{s=1}^r C_{b_s}.
\end{equation}
by Proposition \ref{tilting number} again.
On the other hand, we find that the correspondence $\epsilon \mapsto (b_1,\ldots, b_r)$ defined above provides a bijection between $\{\epsilon \colon [n] \rightarrow \{\pm1\} \mid \epsilon(1)=1\}$ and $\bigcup_{r=1}^n Z_n^r$, where $Z_n^r:=\{(b_1,\ldots, b_r)\in \mathbb{Z}^r \mid 
\text{$\sum_{s=1}^r b_s=n$, $b_s>0$ for all $s=1,\ldots,r$}\}$.
Summarizing these facts, we have 
\begin{equation} \label{Pnr}
\sum_{\substack{ \epsilon \colon [n] \rightarrow \{\pm1\} \\ \epsilon(1)=1}} \# \tilt k\Gamma_{\epsilon}
= \sum_{r=1}^n \sum_{b\in Z_n^r} \prod_{s=1}^r C_{b_s}.
\end{equation}

From the equation (\ref{stau compute}), it reminds to show the following claim. 

\begin{lemma} \label{total sum}
The following equation holds. 
$$
\sum_{r=1}^n \sum_{b\in Z_n^r} \prod_{s=1}^r C_{b_s} = \frac{1}{2}\binom{2n}{n}. 
$$
\end{lemma}

\begin{proof}
Let $P^r_n:= \sum_{b \in Z^r_n} \prod_{s=1}^r C_{b_s}$. 
We show that $\sum_{r=1}^n P_n^r = \frac{1}{2}\binom{2n}{n}$ by induction on $n$. 
\begin{enumerate}
\item[i)] For $n=1$, we have $P_1^1=C_1=1$.
\item[ii)] Assume that $n>1$ and $\sum_{r=1}^mP_m^r = \frac{1}{2}\binom{2m}{m}$ holds for $1\leq m<n$.  

Clearly, we have $P_{n}^1 = C_n$, and for $1<r \leq n$, 
\begin{eqnarray*}
 P_{n}^r = \sum_{\substack{b \in Z_{n}^r}} \prod_{s=1}^r C_{b_s} 
=C_1 P_{n-1}^{r-1}+ C_2 P_{n-2}^{r-1}+ \cdots + C_{n-r+1} P_{r-1}^{r-1}.
\end{eqnarray*}
Therefore, 
\begin{align*}
\sum_{r=1}^n P_{n}^r &= C_n + \sum_{k=1}^{n-1} C_k \left\{ \sum_{r=1}^{n-k} P_{n-k}^r \right\} 
\quad \overset{{\rm induction}}{=} C_n + \frac{1}{2} \sum_{k=1}^{n-1} C_k \binom{2(n-k)}{n-k}\\
& \overset{\rm definition}{=} C_n + \frac{1}{2} \sum_{k=1}^{n-1} (n-k+1) C_k C_{n-k} 
= C_n + \frac{1}{2} \sum_{k=1}^{n-1} (k+1) C_k C_{n-k}.
\end{align*}
Adding last two equations, we have
\begin{equation*}
2\cdot \sum_{r=1}^n P_{n}^r =  2C_n + \frac{1}{2}(n+2) \sum_{k=1}^{n-1} C_k C_{n-k} 
\overset{{\rm Lem.\ref{Catalan}(1)}}{=}  2C_n + \frac{1}{2}(n+2)(C_{n+1} - 2C_n) 
\overset{{\rm Lem.\ref{Catalan}(2)}}{=}  \binom{2n}{n}.       
\end{equation*}
Thus, we obtain the desired equation for $n$. \qedhere
\end{enumerate} 
\end{proof}

We finish the proof of Theorem \ref{line}.
\end{proof}

Next, we enumerate the number of support $\tau$-tilting modules over Brauer cycle algebras. 

\begin{definition} \rm
A {\it Brauaer cycle algebra} having $n$ edges is defined by the bounded quiver algebra 
$A_{\Xi_n}:=k\Gamma'/I'$, where $\Gamma'$ is a quiver given by

$$
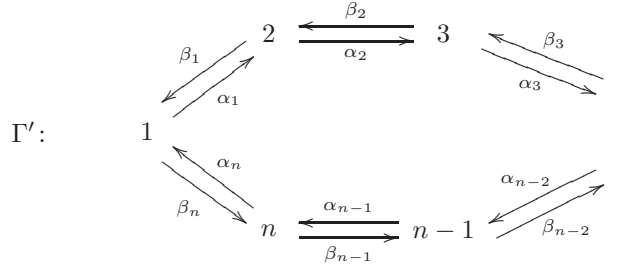
\begin{xy}
(-15, 0)*{\Gamma' \colon}="a",
(0, 0)*{1}="1", 
(16, 13)*{2}="2", (39, 13)*{3}="3", 
(16,-13)*{n}="n", (39, -13)*{n-1}="n-1", 
(64, -0)*{\cdot}="q", (64, 4)*{\cdot}="w", (64, -4)*{\cdot}="e", 
{ (2,4) \ar@{<-}^{\beta_1} (13,12)},
{ (3.5, 2) \ar@{->}_{\alpha_1} (14,10)},
{ (2,-4) \ar@{->}_{\beta_{n}} (13,-12)},
{ (3.5, -2) \ar@{<-}^{\alpha_{n}} (14,-10)},
{ (20,14) \ar@{<-}^{\beta_2} (35,14)},
{ (20, 12) \ar@{->}_{\alpha_2} (35,12)},
{ (20,-14) \ar@{->}_{\beta_{n-1}} (33,-14)},
{ (20, -12) \ar@{<-}^{\alpha_{n-1}} (33,-12)},
{ (45, 13) \ar@{<-}^{\beta_{3}} (60, 7) },
{ (46, -14) \ar@{->}_{\beta_{n-2}} (60, -7) },
{ (44, 11) \ar@{->}_{\alpha_{3}} (59, 5) },
{ (45, -12) \ar@{<-}^{\alpha_{n-2}} (59, -5) }
\end{xy}
$$

and $I'$ is an ideal generated by all elements
\begin{itemize}
\item $\alpha_n \alpha_1, \alpha_1 \alpha_2, \ldots, \alpha_{n-1} \alpha_n,$
\item $\beta_n \beta_{n-1}, \ldots, \beta_2\beta_1, \beta_1 \beta_n,$
\item $(\beta_n \alpha_n)^{m(n)} - (\alpha_1 \beta_1)^{m(1)}, 
(\beta_{1} \alpha_{1})^{m(1)} - (\alpha_{2} \beta_{2})^{m(2)}, 
\ldots, (\beta_{n-1} \alpha_{n-1})^{m(n-1)} - (\alpha_{n} \beta_{n})^{m(n)}$,
\end{itemize}
where $m\colon [n] \rightarrow \mathbb{Z}_{>0}$ is a function.
We say that $A_{\Xi_n}$ is a {\it Brauer odd-cycle} (resp. {\it even-cycle}) {\it algebra} if $n$ is odd (resp. even).
\end{definition}

It is known that a Brauer cycle graph algebra is $\tau$-tilting-finite if it is odd cycle.

\begin{proposition}\cite[Corollary 4.5]{Ada2}
$A_{\Xi_n}$ is $\tau$-tilting-finite if and only if $n$ is odd.
\end{proposition}

If $n$ is even, we can take a map $\epsilon \colon [n] \rightarrow \{\pm1\}$ such that 
$\epsilon(i)=1$ if $i$ is odd; otherwise $-1$, 
then the underlying graph of $\Gamma'_{\epsilon}$ is an extended Dynkin graph 
$\tilde{\mathbb{A}}$. 
On the other hand, if $n$ is odd, then we have the following result.

\begin{lemma} \label{odd}
Assume that $n$ is odd. 
For any map $\epsilon \colon [n] \rightarrow \{\pm1\}$, 
the quiver $\Gamma'_{\epsilon}$ is a disjoint union of odd number of connected quivers of type $\mathbb{A}$. 
\end{lemma}

\begin{proof}
Let $\epsilon \colon [n] \rightarrow \{\pm1\}$ be a map.
Clearly, every connected component of $\Gamma'_{\epsilon}$ is a quiver of Dynkin type $\mathbb{A}$, so we see that the number of these connected components is odd. 

Assume that $\Gamma'_{\epsilon}$ is a disjoint union of $r$ connected quivers. 
In this situation, we have $-\epsilon(1) = (-1)^{r-1}(-1)^{n} \epsilon(1)$. 
Since $n$ is odd, $r$ must be odd.
\end{proof}

For this class of algebras, our result is the following.

\begin{theorem} \label{odd cycle}
Let $A_{\Xi_n}$ be a Brauer odd-cycle algebra having $n$ edges.
Then we have 
\begin{equation*}
\# \twotilt A_{\Xi_n} = 2^{2n-1}.
\end{equation*}
\end{theorem}

We need the following result, that coming from the proof of Lemma \ref{total sum}. 
Now, let $P^r_n:= \sum_{b \in Z^r_n} \prod_{s=1}^r C_{b_s}$ be a sum in Lemma \ref{total sum}. 

\begin{lemma} \label{part}

\begin{eqnarray}\label{odd even}
\sum_{\text{$r:$ odd}} P^r_n = n C_{n-1} \quad \text{\rm and} \quad
\sum_{\text{$r:$ even}} P^r_n = (n-1) C_{n-1}.
\end{eqnarray}
\end{lemma}

\begin{proof}
Let $O_n := \sum_{r \colon \text{\rm {odd}}} P^r_n$ and $E_n := \sum_{r \colon \text{\rm {even}}} P^r_n$. 
We claim $O_n - E_n = C_{n-1}$ and show it by induction on $n$. 
After that, we obtain the desired equation (\ref{odd even}) since we have already known $ O_n + E_n = \sum_{r=1}^n P^r_n= \frac{1}{2} \binom{2n}{n}$ by Lemma \ref{total sum}.
\begin{enumerate}
\item[i)] For $n=1$, we have $O_1 - E_1 = 1-0 =1$. 
\item[ii)] Assume that $n>1$ and $O_m - E_m = C_{m-1}$ holds for $1\leq m<n$. Then we have
\begin{equation*}
O_n := \sum_{\text{\rm $r\colon$odd}} P^r_n = \sum_{k=1}^{n-1} C_k \left\{ \sum_{\text{\rm $r\colon$odd}} P_{n-k, r-1} \right\}  + C_n =\sum_{k=1}^{n-1} C_k E_{n-k} + C_n.   
\end{equation*}
Similarly, we have 
\begin{eqnarray*} 
E_n  := \sum_{r \colon \text{\rm {even}}} P^r_n &=& \sum_{k=1}^{n-1} C_k \left\{ \sum_{\text{\rm $r:$ even}} P_{n-k, r-1} \right\}  = \sum_{k=1}^{n-1} C_k O_{n-k} \\ 
&=& \sum_{k=1}^{n-1} C_k (E_{n-k} + C_{n-k-1})  = 
\sum_{k=1}^{n-1} C_kE_{n-k} + C_n - C_{n-1}
\end{eqnarray*}
Therefore, we have the desired equation $O_n - E_n = C_{n-1}$ for $n$.\qedhere
\end{enumerate}
\end{proof}

Now, we are ready to prove Theorem \ref{odd cycle}.

\begin{proof}[Proof of Theorem \ref{odd cycle}]
Let $n$ be an positive odd integer and $A_{\Xi_n}$ be a Brauer odd-cycle algebra having $n$ edges.
Let $B_n:=k\Gamma'/I''$ be an algebra with radical square zero, where $I''$ is an ideal generated by all path of length $2$.
From the similar reason as the case of Brauer line algebras, 
there exists an ideal $J'$ in $A_{\Xi_n}$ such that $A_{\Xi_n}/J' \cong B_n$ and 
satisfying the assumption of Proposition \ref{reduction}. 
Therefore, we have $\#\stau A_{\Xi_n} = \#\stau B_n < \infty$.
In the following, we calculate the number $ \#\stau B_n$ instead.

For each map $\epsilon \colon [n] \rightarrow \{\pm1\}$, 
the quiver $\Gamma'_{\epsilon}$ is a disjoint union of odd number of quivers of type $\mathbb{A}$ by Lemma \ref{odd}. 
Clearly, the quiver $\Gamma'_{-\epsilon}$ associated to a map $-\epsilon$ 
is isomorphic to the opposite quiver of $\Gamma'_{\epsilon}$. 
In particular, they have the same underlying graph. So, we have
\begin{equation} \label{stau compute2}
\# \stau B_n = \sum_{\epsilon \colon [n] \rightarrow \{\pm1\}} \# \tilt k\Gamma'_{\epsilon} = 2 \sum_{\substack{ \epsilon \colon [n] \rightarrow \{\pm1\}\\ \epsilon(1)=1}} \# \tilt k\Gamma'_{\epsilon} .
\end{equation}
by Theorem \ref{stau-tilt-tilt}. 

Next, we observe the shape of $\Gamma'_{\epsilon}$ more detail.
Let $\epsilon \colon [n] \rightarrow \{\pm1\}$ be a map such that $\epsilon(1)=1$. 
Let $Q$ be a connected component of $\Gamma'_{\epsilon}$ containing the vertex $1$. Note that this is a quiver of Dynkin type $\mathbb{A}$.
Suppose that $t$ is the number of vertices of $Q$.
\begin{enumerate}
\item[(a)] If $t=n$, then we have $Q= \Gamma_{\epsilon}'$. 
In this case, we have $\# \tilt k\Gamma_{\epsilon} = \# \tilt kQ = C_n$. 

\item[(b)] Assume that $t=n-1$. Let $i \notin Q$ be the remained vertex of $\Gamma'_{\epsilon}$. 
If $\epsilon(i)\neq \epsilon(i-1)$ or $\epsilon(i)\neq \epsilon(i-1)$ hold, then 
$Q$ contains the vertex $i$: it is a contradiction. 
Thus, they must satisfy $\epsilon(i-1)=\epsilon(i)=\epsilon(i+1)$.
However, we have $\epsilon(i+1)= (-1)^{n}\epsilon(i-1) \neq \epsilon(i-1)$ since $n$ is odd. 
Therefore, $t \neq n-1$. 

\item[(c)] Assume that $1\leq t \leq n-2$. 
By Lemma \ref{odd}, the remained part $\Gamma_{\epsilon} \backslash Q$ is a disjoint union of even number of connected quivers $Q_{b_1}, \ldots ,Q_{b_r}$ of type $\mathbb{A}$ having $b_1, \ldots, b_r$ vertices respectively, where $r$ is even. 
In this case, we have 
$\# \tilt k \Gamma_{\epsilon} = \# \tilt kQ  \prod_{s=1}^r \# \tilt kQ_{b_s} = 
C_t \prod_{s=1}^r C_{b_s}$.
\end{enumerate}

Running over all maps $\epsilon \colon [n] \rightarrow \{\pm1\}$ with $\epsilon(1)=1$, we have 
\begin{equation}
\sum_{\substack{\epsilon \colon [n]\rightarrow \{\pm1\} \\ \epsilon(1)=1}} \# \tilt k\Gamma_{\epsilon} = nC_n + 
\sum_{t=1}^{n-2} tC_t \left\{ \sum_{\text{\rm $r\colon$even}} P_{n-t,r} \right\} 
\overset{{\rm Lem.\ref{part}.}}{=} nC_n + \sum_{t=1}^{n-2} tC_t (n-t-1) C_{n-t-1} . 
\end{equation}

However, the last term is equal to
$$
nC_n + \sum_{t=0}^{n-1} (t+1)C_{t}(n-t)C_{n-t-1} - \sum_{t=0}^{n-1}C_tC_{n-t-1} =
nC_n + \sum_{t=0}^{n-1} \binom{2(t-1)}{t-1}\binom{2(n-t-1)}{n-t-1} -nC_n = 4^{n-1},
$$ 
where we use Lemma \ref{Catalan}(2) and (3). 

Consequently, we get 
$$
\#\stau A_{\Xi_n} = \#\stau B_n = 2 \sum_{\substack{\epsilon \colon [n]\rightarrow \{\pm1\} \\ \epsilon(1)=1}} \# \tilt k\Gamma_{\epsilon} 
= 2\cdot 4^{n-1} =2^{2n-1}.
$$
We complete the proof.
\end{proof}

\end{document}